\documentclass{svjour3}                     % onecolumn (standard format)
\smartqed  % flush right qed marks, e.g. at end of proof
\usepackage{enumitem}
\usepackage{graphicx}
\usepackage{epstopdf}
\usepackage{url, booktabs, color}
\usepackage{amsmath, amsfonts}
\graphicspath{{./Graphs/}}

\newcommand{\R}{{\mathbb{R}}}
\renewcommand{\P}{{\mathbb{P}}}
\newcommand{\E}{{\mathbb{E}}}
\newcommand{\bx}{\mathbf{x}}
\newcommand{\bc}{\mathbf{c}}
\newcommand{\by}{\mathbf{y}}
\newcommand{\bepsilon}{\boldsymbol{\epsilon}}
\newcommand{\be}{\mathbf{e}}
\newcommand{\bzero}{\mathbf{0}}
\newcommand{\bA}{\mathbf{A}}
\newcommand{\bW}{\mathbf{W}}
\newcommand{\F}{\mathcal{F}}
\newcommand{\A}{\mathcal{A}}
\newcommand{\ba}{\mathbf{a}}
\newcommand{\bb}{\mathbf{b}}
\newcommand{\bxhat}{\mathbf{\hat{x}}}
\newcommand{\bN}{\mathbf{N}}
\newcommand{\bw}{\mathbf{w}}
\newcommand{\bd}{\mathbf{d}}
\newcommand{\W}{\mathcal{W}}
\newcommand{\btheta}{{\boldsymbol{\theta}}}
\newcommand{\bphi}{\boldsymbol{\phi}}
\renewcommand{\H}{\mathcal{H}}
\newcommand{\bt}{\mathbf{t}}
\newcommand{\balpha}{\boldsymbol{\alpha}}
\newcommand{\bK}{\mathbf{K}}
\DeclareMathOperator{\VI}{\text{VI}}
\newcommand{\blue}{\textcolor{black}}
\newcommand{\blockblue}{\color{black}}
\newcommand{\assref}[1]{\textbf{A\ref{#1}}}

\spnewtheorem{assumption}{Assumption}{\bf}{\rm}
\spnewtheorem*{assumption*}{Assumption}{\bf}{\rm}

 \journalname{Mathematical Programming}
\begin{document}

\title{Data-Driven Estimation in Equilibrium using Inverse Optimization}
%\subtitle{Do you have a subtitle?\\ If so, write it here}
\titlerunning{Data-Driven Estimation in Equilibrium}       

\author{Dimitris Bertsimas         \and
        	    Vishal Gupta 	\and
	    Ioannis Ch. Paschalidis
}

\authorrunning{D. Berstimas, V. Gupta \& I. Ch. Paschalidis} % if too long for running head

\institute{Dimitris Bertsimas \at
              MIT, Sloan School of Management \\
              Massachusetts Institute of Technology \\
              Cambridge, MA 02139 \\
              \email{dbertsim@mit.edu}           %  \\
           \and
           Vishal Gupta  \at
	Operations Research Center \\
	Massachusetts Institute of Technology \\
	Cambridge, MA 02139 \\
	\email{vgupta1@mit.edu}
	\and
	Ioannis Ch. Paschalidis \at
	Department of Electrical and Computer Engineering \\
	Boston University \\
	 Boston, MA 02215 \\
	  \email{yannisp@bu.edu}	
}

\date{Received: date / Accepted: date}
% The correct dates will be entered by the editor

\maketitle

\begin{abstract}
Equilibrium modeling is common in a variety of fields such as game theory and transportation science.  The inputs for these models, however, are often difficult to estimate, while their outputs, i.e., the equilibria they are meant to describe,  are often directly observable.  By combining ideas from inverse optimization with the theory of variational inequalities, we develop an efficient, data-driven technique for estimating the parameters of these models from observed equilibria.  
We use this technique to estimate the utility functions of players in a game from their observed actions and to estimate the congestion function on a road network from traffic count data.  
A distinguishing feature of our approach is that it supports both parametric and \emph{nonparametric} estimation by leveraging ideas from statistical learning (kernel methods and regularization operators).  
In computational experiments involving Nash and Wardrop equilibria in a nonparametric setting, we find that a) we effectively estimate the unknown demand or congestion function, respectively, and b) our proposed regularization technique substantially improves the out-of-sample performance of our estimators.  
% \PACS{PACS code1 \and PACS code2 \and more}
% \subclass{MSC code1 \and MSC code2 \and more}
\end{abstract}

\section{Introduction}

Modeling phenomena as equilibria is a common approach in a variety of fields.  Examples include Nash equilibrium in game theory, traffic equilibrium in transportation science and  market equilibrium in economics.  
Often, however, the model primitives or ``inputs" needed to calculate equilibria are not directly observable and can be difficult to estimate.  Small errors in these estimates may have large impacts on the resulting equilibrium.  
This problem is particularly serious in \emph{design} applications, where one seeks to (re)design a system so that the induced equilibrium satisfies some desirable properties, such as maximizing social welfare.  In this case, small errors in the estimates may substantially affect the optimal design.  Thus, developing accurate estimates of the primitives is crucial. 

In this work we propose a novel framework to estimate the unobservable model primitives for systems in equilibrium.  Our data-driven approach hinges on the fact that although the model primitives may be unobservable, it is frequently possible to observe equilibria experimentally.  We use these observed equilibria to estimate the original primitives.  

We draw on an example from game theory to illustrate.  Typically, one specifies the utility functions for each player in a game and then calculates Nash equilibria. In practice, however, it is essentially impossible to observe utilities directly.  Worse, the specific choice of utility function often makes a substantial difference in the resulting equilibrium.  Our approach amounts to estimating a player's utility function from her actions in previous games, assuming her actions were approximately equilibria with respect to her opponents.  In contrast to her utility function, her previous actions \emph{are} directly observable.  This utility function can be used either to predict her actions in future games, or as an input to subsequent mechanism design problems involving this player in the future.  

A second example comes from transportation science.  Given a particular road network, one typically specifies a cost function and then calculates the resulting flow under user (Wardrop) equilibrium.  However, measuring the cost function directly in a large-scale network is challenging because of the interdependencies among arcs.  Furthermore, errors in estimates of cost functions can have severe and counterintuitive effects; Braess paradox (see \cite{Braess01112005}) is one well-known example.  
Our approach amounts to estimating cost functions using current traffic count data (flows) on the network, assuming those flows are approximately in equilibrium.  Again, in contrast to the cost function, traffic count data are readily observable and frequently collected on many real-life networks.  Finally, our estimate can be used either to predict congestion on the network in the future, or else to inform subsequent network design problems.

In general, we focus on equilibria that can be modeled as the solution to a variational inequality (VI).  VIs are a natural tool for describing equilibria with examples spanning economics, transportation science, physics,  differential equations,  and optimization.  (See Section~\ref{Sec:Examples} or \cite{harker1990finite} for detailed examples.)  Our model centers on solving an \emph{inverse variational inequality problem}: given data that we believe are equilibria, i.e., solutions to some VI, estimate the function which describes this VI, i.e., the model primitives.

Our formulation and analysis is motivated in many ways by the inverse optimization literature.  In inverse optimization, one is given a candidate solution to an optimization problem and seeks to characterize the cost function or other problem data that would make that solution (approximately) optimal.  See \cite{heuberger2004inverse} for a survey of inverse combinatorial optimization problems, \cite{ahuja2001inverse} for the case of linear optimization and \cite{iyengar2005inverse} for the case of conic optimization.  
The critical difference, however, is that we seek a cost function that would make the observed data equilibria, not optimal solutions to an optimization problem.  \blue{In general, optimization problems can be reformulated as variational inequalities (see Sec.~\ref{Sec:Examples}), so that our inverse VI problem \emph{generalizes} inverse optimization, but this generalization allows us to address a variety of new applications.}  
%(The approach in \cite{bertsimas2011inverse} lies between these two extremes, where the data are derived as a consequence of equilibrium assumptions, but used in an inverse optimization framework.)  

To the best of our knowledge, we are the first to consider \blue{inverse variational inequality problems.  Previous work, however, has examined the problem of estimating parameters for systems assumed to be in equilibrium, most notably the structural estimation literature in econometrics and operations management
(\cite{rust1994structural}, \cite{bajari2007estimating}, \cite{allon2011much}, \cite{nevo2001measuring}).}  Although there are a myriad of techniques collectively referred to as structural estimation, roughly speaking, they entail (1) assuming a parametric model for the system including probabilistic assumptions on random quantities, (2) deducing a set of necessary (structural) equations for unknown parameters, and, finally, (3) \blue{solving a constrained optimization problem corresponding to a generalized method of moments (GMM) estimate for the parameters.  The constraints of this optimization problem include the structural equations and possibly other application-specific constraints, e.g., orthogonality conditions of instrumental variables.  
%A great deal of problem-specific insight is usually required in Steps 2 and 3 of structural estimation to ensure the estimates are well-behaved.  
%Indeed, in some applications, authors substantially adapt the above steps as needed.  
Moreover, this optimization problem is typically difficult to solve numerically, as it is can be non-convex with large flat regions and multiple local optima (see \cite{allon2011much} for some discussion).}

{
\blockblue
Our approach differs from structural estimation and other specialized approaches in a number of respects.  From a philosophical point of view, the most critical difference is in the objective of the methodology.  Specifically, in the structural estimation paradigm, one posits a ``ground-truth" model of a system with a known parametric form.  The objective of the method is to learn the parameters in order to provide insight into the system.  By contrast, in our paradigm, we make no assumptions (parametric or nonparametric) about the true mechanics of the system; we treat is as a ``black-box."  Our objective is to fit a model -- in fact, a $\VI$ -- that can be used to predict the behavior of the system.   We make no claim that this fitted model accurately reflects ``reality," merely that it has good predictive power.  

This distinction is subtle, mirroring the distinction between ``data-modelling" in classical statistics and ``algorithmic modeling" in machine learning.  (A famous, albeit partisaned, account of this distinction is \cite{breiman2001statistical}.)  Our approach is kindred to the machine learning point of view.  For a more detailed discussion, please see Appendix~\ref{sec:Casting}.

This philosophical difference has a number of \emph{practical} consequences:
\begin{enumerate}
\item \textbf{Minimal Probabilistic Assumptions:}  Our method has provably good performance in a very general setting with minimal assumptions on the underlying mechanism generating the data.  (See Theorems~\ref{thm:Campi}-\ref{thm:Predictive} for precise statements.) By contrast, other statistical methods, including structural estimation, require a full-specification of the data generating mechanism and can yield spurious results if this specification is inaccurate.  

\item \textbf{Tractability:}  Since our fitted model need not correspond exactly to the underlying system dynamics, we have considerably more flexibility in choosing its functional form.  For several interesting choices, including nonparametric specifications (see next point), the resulting inverse $\VI$ problem can be reformulated as a conic optimization problem.  Conic optimization problems are both theoretically and numerically tractable, even for large scale instances (\cite{boyd2004convex}), in sharp contrast to the non-convex problems that frequently arise in other methods.  

\item \textbf{Nonparametric Estimation:}  Like existing methods in inverse optimization and structural estimation, our approach can be applied in a parametric setting.  Unlike these approaches, our approach also extends naturally to a nonparametric description of the function $\mathbf{f}$ defining the $\VI$.  To the best of our knowledge, existing methods  do not treat this possibility.  Partial exceptions are \cite{NBERw15276} and \cite{ECTA:ECTA123} which use nonparametric estimators for probability densities, but parametric descriptions of the mechanism governing the system.  
The key to our nonparametric approach is to leverage kernel methods from statistical learning to reformulate the infinite dimensional inverse variational inequality problem as a finite dimensional, convex quadratic optimization problem.  
%By contrast, our approach supports both parametric and non-parametric formulation for $\mathbf{f}$.  Our non-parametric approach leverages kernel methods from statistical learning.  See \cite{smola1998learning}, \cite{evgeniou2000regularization} or \cite{trevor2001elements} for a survey of kernel methods.  
%Specifically, we formulate a convex, quadratic optimization problem to solve the nonparametric inverse variational inequality problem.  Open-source and commercial software are available for solving large scale quadrate programming (e.g., ILOG CPLEX, Gurobi, SeDuMi, SDPT3, etc.).  
In applications where we may not know, or be willing to specify a particular form for $\mathbf{f}$ we consider this non-parametric approach particularly attractive.
\end{enumerate}

Although there are  other technical differences between these approaches -- for example, some structural estimation techniques can handle discrete features while our method applies only to continuous problems -- we feel that the most important difference is the aforementioned intended purpose of the methodology.  We see our approach as complementary to existing structural estimation techniques and believe in some applications practitioners may prefer it for its computational tractability and relatively fewer modeling assumptions.  Of course, in applications where the underlying assumptions of structural estimations or other statistical techniques are valid, those techniques may yield potentially stronger claims about the underlying system. 
}

We summarize our contributions below:
\begin{enumerate}
	\item We propose the inverse variational inequality problem to model inverse equilibrium.  We illustrate the approach by estimating market demand functions under Bertrand-Nash equilibrium and by estimating the congestion function in a traffic equilibrium.  
	\item We formulate an optimization problem to solve a parametric version of the inverse variational inequality problem.  The complexity of this optimization depends on the particular parametric form of the function to be estimated.  We show that for several interesting choices of parametric form, the parametric version of the inverse variational inequality problem can be reformulated as a simple conic optimization problem.   
	\item We formulate and solve a nonparametric version of the inverse variational inequality problem using kernel methods.  We show that this problem can be efficiently solved as a convex quadratic optimization problem whose size scales linearly with the number of observations.  
	\item \blue{Under very mild assumptions on the mechanism generating the data, we show that both our parametric and non-parametric formulations enjoy a strong generalization guarantee similar to the guarantee enjoyed by other methods in machine learning.  Namely, if the fitted VI explains the existing data well, it will continue to explain new data well.  Moreover, under some additional assumptions on the optimization problem, equilibria from the VI serve as good predictions for new data points.}   
	\item We provide computational evidence in the previous two examples -- demand estimation under Nash equilibrium and congestion function estimation under traffic equilibrium -- that our proposed approach recovers reasonable functions with good generalization properties and predictive power. 
We believe these results may merit independent interest in the specialized literature for these two applications.
\end{enumerate}

The remainder of this paper is organized as follows.  Section~\ref{Sec:VIs} reviews background material on equilibrium modeling through VIs.  Section~\ref{sec:InverseVI} formally defines the inverse variational inequality problem and solves it in the case that the function to be estimated has a known parametric form.  In preparation for the nonparametric case, Section~\ref{Sec:Kernels} reviews some necessary background material on kernels.  Section~\ref{Sec:NonParametric} formulates and solves the nonparametric inverse variational inequality problem using kernels, and Section~\ref{sec:extensions} \blue{illustrates how to incorporate priors, semi-parametric modeling and ambiguity sets into this framework.  Section~\ref{sec:Generalization} states our results on the generalization guarantees and predictive power of our approach.}  Finally, Section~\ref{sec:Numerics} presents some computational results, and Section~\ref{Sec:Conclusion} concludes.  In the interest of space, almost all proofs are placed in the Appendix.   

In what follows we will use boldfaced capital letters ( e.g., $\bA, \bW$)  to denote matrices, boldfaced lowercase letters (e.g., $\bx, \mathbf{f}( \cdot )$) to denote vectors or vector-valued functions, and ordinary lowercase letters to denote scalars.  We will use caligraphic capital letters (e.g., $\mathcal{S}$) to denote sets.   For any proper cone $C$, i.e. $C$ is pointed, closed, convex and has a strict interior, we will say $ \bx \leq_C \by$ whenever $\by - \bx \in C$. 

%%%%%%%%%%%%%%%%%%%%%%%%%%%
\section{Variational Inequalities: Background}
\label{Sec:VIs}
\subsection{Definitions and Examples}
\label{Sec:Examples}
 \label{Sec:Traffic} 
In this section, we briefly review some results on variational inequalities that we use in the remainder of the paper.  For a more complete survey, see \cite{harker1990finite}.  

Given a function $\mathbf{f} : \R^n \rightarrow \R^n$ and a non-empty set $\F \subseteq \R^n$ the variational inequality problem, denoted $\VI( \mathbf{f}, \F)$, is to find an $\bx^* \in \F$ such that 
\begin{equation}
\label{eq:DefVI}
\mathbf{f}( \bx^*)^T ( \bx - \bx^*) \geq 0, \quad \forall \bx \in \F.
\end{equation}
%We say that $\bx^*$ is an exact solution to $\VI(\mathbf{f}, \F)$.  
A solution $\bx^*$ to $\VI( \mathbf{f}, \F)$ need not exist, and when it exists, it need not be unique.  We can guarantee the existence and uniqueness of the solution by making appropriate assumptions on $\mathbf{f}( \cdot )$ and/or $\F$, e.g., $\mathbf{f}$ continuous and $\F$ convex and compact.   See \cite{harker1990finite} for other less stringent conditions.     

There are at least three classical applications of VI modeling that we will refer to throughout the paper: constrained optimization, Nash Equilibrium, and Traffic (or Market) Equilibrium.  

%%%%%%%%%%%%%%%%%%%%%%%%%%%
\emph{\textbf{Constrained Optimization.}}
The simplest example of a VI is in fact not an equilibrium, per se, but rather convex optimization.  
Nonetheless, the specific example is very useful in building intuition about VIs.    
Moreover, using this formalism, one can derive many of the existing results in the inverse optimization literature as a special case of our results for inverse VIs in Section~\ref{Sec:Parametric}.  
 
Consider the problem 
\begin{equation}\label{eq:ConvexOpt}
\min_{\bx \in \F} F( \bx ).  
\end{equation}
The first order necessary conditions for an optimal solution of this problem are (see, e.g., \cite{bertsekas1999nonlinear})
\begin{equation} \label{eq:ConvexOptVI}
\nabla F( \bx^*)^T ( \bx - \bx^*) \geq 0, \quad \forall \bx \in \F.
\end{equation}
These conditions are sufficient in the case that $F$ is a convex function and $\F$ is a convex set.  Observe, then, that solving \eqref{eq:ConvexOpt} is equivalent to finding a point which satisfies Eq. \eqref{eq:ConvexOptVI}, which is equivalent to solving $\VI(\nabla F, \F)$. 

Note that, in general, a VI with a function $\mathbf{f}$ whose Jacobian is symmetric models an optimization problem (see \cite{harker1990finite}).  

%%%%%%%%%%%%%%%%%%%%%%%%%%%
\emph{\textbf{Nash Equilibrium.}}
Our first application of VI to model equilibrium is non-cooperative Nash equilibrium.   
Consider a game with $p$ players.  Each player $i$ chooses an action from a set of feasible actions, $\ba_i \in \A_i \subseteq \R^{m_i}$, and receives a utility 
$U_i(\ba_1, \ldots, \ba_p )$.  Notice in particular, that player $i$'s payoff may depend upon the actions of other players.  We will assume that $U_i$ is differentiable and concave in $\ba_i$ for all $i$ and that $\A_i$ is convex for all $i$.     

A profile of actions for the players $(\ba_1^*, \ba_2^*, \ldots \ba_p^*)$ is said to be a Nash Equilibrium if no single player can unilaterally change her action and increase her utility.  See \cite{fudenberg1991game} for a more complete treatment.  In other words, player $i$ plays her best response given the actions of the other players.  More formally, 
\begin{equation}  
\label{eq:NashCondition}
\ba_i^* \in \arg \max_{\ba \in \A_i} U_i( \ba_1^*, \ldots, \ba_{i-1}^*, \ba, \ba_{i+1}^*, \ldots, \ba^*_p), \quad i = 1, \ldots, p.
\end{equation}

This condition can be expressed as a VI.  Specifically, a profile $\ba^* = (\ba_1^*, \ba_2^*, \ldots \ba^*_p)$ is a Nash Equilibrium, if and only if it solves $\VI( \mathbf{f}, \F)$ where $\F = \A_1 \times \A_2 \times \dots \times \A_p $,  
\begin{equation} 
\label{eq:NashVI}
\mathbf{f} ( \ba ) = \begin{pmatrix} 
				-\nabla_1 U_1 ( \ba  ) 
				\\
				\vdots 
				\\ 
				-\nabla_p U_p( \ba )
		\end{pmatrix}
\end{equation}
and $\nabla_i$ denotes the gradient with respect to the variables $\ba_i$  
(see \cite{harker1990finite} for a proof.)  

It is worth pointing out that many authors use Eq. \eqref{eq:NashCondition} to conclude 
\begin{equation}
\label{eq:NashBad}
\nabla_i U_i( \ba^*_1, \ldots, \ba^*_p) = \bzero, \quad i = 1, \ldots, p,
\end{equation}
where $\nabla_i$ refers to a gradient with respect to the coordinates of $\ba_i$.  
This characterization assumes that each player's best response lies on the strict interior of her strategy set $\A_i$.  The assumption is often valid, usually because the strategy sets are unconstrained.  Indeed, this condition can be derived as a special case of \eqref{eq:NashVI} in the case $\A_i = \R^{m_i}$.  In some games, however, it is not clear that an equilibrium must occur in the interior, and we must use \eqref{eq:NashVI} instead.  We will see an example in Sec. \ref{Sec:Parametric}.
 
%%%%%%%%%%%%%%%%%%%%%%%%%%%
\emph{\textbf{Wardrop Equilibrium.}} 
Our final example of a VI is Wardrop or user-equilibrium from transportation science.  
 Wardrop equilibrium is extremely close in spirit to the market (Walrasian) equilibrium model in economics -- see \cite{dafermos1984network}, \cite{zhao1991general}  -- and our comments below naturally extend to the Walrasian case.  

Specifically, we are given a directed network of nodes and arcs $(\mathcal{V}, \A)$, representing the road network of some city.  Let $\bN \in \{0, 1 \}^{|\mathcal{V} | \times | \A | }$ be the node-arc incidence matrix of this network.  For certain pairs of nodes $\bw = (w_s, w_t) \in \mathcal{W}$, we are also given an amount of flow $d^\bw$ that must flow from $w_s$ to $w_t$.  The pair $\bw$ is referred to as an origin-destination pair.  
Let $\bd^{\bw} \in \R^{|V|}$ be the vector which is all zeros, except for a $(-d^\bw)$ in the coordinate corresponding to node $w_s$ and a $(d^\bw)$ in the coordinate corresponding to node $w_t$.  

We will say that a vector of flows $\bx \in \R^{| \A |}_+$ is feasible if $\bx \in \F$ where 
\[
\F = \left\{ \bx : \exists \bx^\bw \in \R^{ |\A|}_+ \text{ s.t. } 
\quad
\bx  = \sum_{\bw \in \W} \bx^\bw, 
\quad 
\bN \bx^\bw = \bd^\bw \quad \forall \bw \in \mathcal{W} \right\}.
\]

Let $c_a: \R^{| \A |}_+ \rightarrow \R_+$ be the ``cost" function for arc $a \in \A$.  The interpretation of cost, here, is deliberately vague.  The cost function might represent the actual time it takes to travel an arc, tolls users incur along that arc, disutility from environmental factors along that arc, or some combination of the above.  Note that because of interdependencies in the network, the cost of traveling arc $a$ may depend not only on $\bx_a$, but on the flows on other arcs as well.  Denote by $\bc ( \cdot )$ the vector-valued function whose $a$-th component is $c_a( \cdot )$.  

A feasible flow $\bx^*$ is a Wardrop equilibrium if for every origin-destination pair $\bw \in W$, and any path connecting $(w_s, w_t)$ with positive flow in $\bx^*$, the cost of traveling along that path is less than or equal to the cost of traveling along any other path that connects $(w_s, w_t)$.  Here, the cost of traveling along a path is the sum of the costs of each of its constituent arcs.  Intuitively, a Wardrop equilibrium captures the idea that if there exists a less congested route connecting $w_s$ and $w_t$, users would find and use it instead of their current route.  

It is well-known that a Wardrop equilibrium is a solution to $\VI(\bc, \F)$.  
%Note that a particular challenge of this VI is its scale; a typical instance may have tens of thousands of arcs and hundreds of thousands of origin-destination pairs.     

%%%%%%%%%%%%%%%%%%%%%%%%%%%%%%%%%%%%%    
\subsection{Approximate Equilibria}
\label{Sec:ApproxEquil}
%We next propose a notion of an approximate equilibrium.  
Let $\epsilon > 0$.  We will say that $\bxhat \in \F$ is an $\epsilon$-approximate solution to $\VI(\mathbf{f}, \F)$ if 
\begin{equation}
\label{eq:defApproxEquil}
\mathbf{f}(\bxhat)^T(\bx - \bxhat) \geq - \epsilon, \quad \forall \bx \in \F.
\end{equation}
This notion of an approximate solution is not new to the VI literature-- it corresponds exactly to the condition that the primal gap function of the VI is bounded above by $\epsilon$ and is frequently used in the analysis of numerical procedures for solving the VI.  We point out that $\epsilon$-approximate solutions also frequently have a modeling interpretation.  For example, consider the case of constrained convex optimization (cf. Eq. \eqref{eq:ConvexOpt}).  Let $\bx^*$ be an optimal solution.  Since $F$ is convex, we have $F( \bxhat) - F(\bx^*) \leq -\nabla F( \bxhat )^T ( \bx^* - \bxhat) \leq \epsilon$.  In other words, $\epsilon$-approximate solutions to VIs generalize the idea of $\epsilon$-optimal solutions to convex optimization problems.  
Similarly, in a Nash equilibrium, an $\epsilon$-approximate solution to the VI \eqref{eq:NashVI} describes the situation where each player $i$ does not necessarily play her best response given what the other players are doing, but plays a strategy which is no worse than $\epsilon$ from her best response.

The idea of $\epsilon$-approximate solutions is not the only notion of an approximate equilibrium.  An alternative notion of approximation is that $\| \bxhat - \bx^*\| \leq \delta$ where $\bx^*$ is a solution to the $\VI(\mathbf{f}, \F)$.  We say such a $\bxhat \in \F$ is $\delta$-near a solution to the $\VI(\mathbf{f}, \F)$.  As shown in Theorem~\ref{thm:PangApprox}, these two ideas are closely related.  The theorem was proven in \cite{pang1987posteriori} to provide stopping criteria for certain types of iterative algorithms for solving VIs.  We reinterpret it here in the context of approximate equilibria.

Before stating the theorem, we define strong monotonicity.  We will say that $\mathbf{f} ( \cdot )$ is \emph{strongly monotone} if $\exists \gamma > 0$ such that
\[
(\mathbf{f}( \bx) - \mathbf{f}(\by) )^T(\bx - \by ) \geq \gamma \| \bx - \by \|^2  \quad \forall \bx, \by \in \F.
\]
\blue{When the VI corresponds to constrained optimization (cf. Eqs.~\eqref{eq:ConvexOpt}, \eqref{eq:ConvexOptVI}), strong monotonicity of $f$ corresponds to strong convexity of $F$.  Intuitively, strong monotonicity ensures that $f$ does not have large, flat regions.}
\begin{theorem}[\cite{pang1987posteriori}]
\label{thm:PangApprox}
Suppose $\mathbf{f}$ is strongly monotone with parameter $\gamma$.  Then every $\epsilon$-approximate solution to $\VI(\mathbf{f}, \F)$ is $\sqrt{\frac{\epsilon}{\gamma}}$-near an exact solution.  
\end{theorem}

We require Theorem~\ref{thm:PangApprox} in Section~\ref{sec:Generalization} to prove some of our generalization results.  
 
%%%%%%%%%%%%%%%%%%%%%%%%%%%
\subsection{Characterizing Approximate Solutions to VIs over Conic Representable Sets}
\label{Sec:CharacterizingApprox}
In this section we provide an alternative characterization of an $\epsilon$-approximate solution (cf. Eq. \eqref{eq:defApproxEquil}) in the case when $\F$ is represented by the intersection of conic inequalities.  

Specifically, for the remainder of the paper, we will assume: 
\begin{assumption} \label{assumption1}
$\F$ can be represented as the intersection of a small number of conic inequalities in standard form,
$
\F = \{ \bx : \bA \bx = \bb, \bx \in C \}.
$
\end{assumption}
\begin{assumption}\label{assumption2}
$\F$ satisfies a Slater-condition
\end{assumption}
The assumption that $\F$ is given in standard form is not crucial.  All of our results extend to the case that $\F$ is not given in standard form at the expense of some notation.  It is, however, crucial, that $\F$ is conic representable.  
Observe that when $C$ is the nonnegative orthant, we recover the special case where $\F$ is a polyhedron.  With other choices of $C$, e.g., the second-order cone, we can model more complex sets, such as intersection of ellipsoids.  To stress the dependence on $\bA, \bb, C$, we will write $\VI(\mathbf{f}, \bA, \bb, C).$

The following result was first proven in \cite{Aghassi2006481} to describe a reformulation of $\VI(\mathbf{f}, \bA, \bb, C)$ as a single-level optimization problem.  We reinterpret here as a characterization of approximate equilibria and sketch a short proof for completeness.  

\begin{theorem}[\cite{Aghassi2006481}]
\label{thm:Aghassi}
Under assumptions \assref{assumption1}, \assref{assumption2}, the solution $\bxhat$ is an $\epsilon$-approximate equilibrium to $\VI(\mathbf{f}, \bA, \bb, C)$ if and only if $\exists \by$ s.t.
\begin{align} 
\label{eq:DualFeas}
	\bA^T \by \leq_C \mathbf{f}( \bxhat), 
\\ \label{eq:StrongDuality}
	\mathbf{f}(\bxhat)^T \bxhat  - \bb^T \by \leq \epsilon.
\end{align}	
\end{theorem}
\begin{proof}
First suppose that $\bxhat$ is an $\epsilon$-approximate equilibrium.  Then, from Eq. \eqref{eq:defApproxEquil}, 
\[
\mathbf{f}( \bxhat)^T \bxhat - \epsilon \leq \mathbf{f}( \bxhat)^T \bx, \quad \forall \bx \in \F,
\] 
which is equivalent to 
$
\mathbf{f}( \bxhat)^T \bxhat - \epsilon \leq \min_{\bx \in \F} \mathbf{f}( \bxhat)^T \bx.
$
The right hand side is a conic optimization problem in $\bx$, and the above shows it is bounded below.  Since $\F$ has non-empty interior, strong duality holds (see \cite{boyd2004convex}), which implies that there exists a dual solution $\by$ that attains the optimum.  In other words,
\[
\min_{\bx \in \F} \mathbf{f}( \bxhat)^T \bx = \max_{\by : \bA^T \by \leq_C \mathbf{f}( \bxhat) } \bb^T \by.
\]
 Substituting this dual solution into the above inequality and rearranging terms yields the result.  The reverse direction is proven analogously using weak conic duality.  
\qed
\end{proof}
The above proof leverages the fact that the duality gap between an optimal primal and dual solution pair is zero.  We can instead formulate a slightly different characterization by leveraging complementary slackness.  In this case, Eq. \eqref{eq:StrongDuality} is replaced by the additional constraints
\begin{align}
\label{eq:CompSlack}
\sum_{i=1}^n x_i(f_i( \bxhat) - \by^T \bA \be_i) &\leq \epsilon.
\end{align}
Depending on the application, either the strong duality representation (cf. Eqs. \eqref{eq:DualFeas}, \eqref{eq:StrongDuality}) or the complementary slackness representation  (cf. Eqs. \eqref{eq:DualFeas}, \eqref{eq:CompSlack}  may be more natural.  We will use 
the strong duality formulation in Section~\ref{sec:TrafficEstimation}
and
the the complementary slackness formulation in Section~\ref{sec:DemandEstimation}.  

%%%%%%%%%%%%%%%%%%%%%%%%%%%
\section{The Inverse Variational Inequality Problem}
\label{sec:InverseVI}
\subsection{Problem Formulation}
We are now in a position to pose the inverse variational inequality problem.  
We are given observations $(\bx_j, \bA_j, \bb_j, C_j)$ for $j = 1, \ldots, N$.  In this context, we modify Assumption~\assref{assumption2} to read 
\begin{assumption*}
	The set $\mathcal{F}_j = \{ \bx \in \R^n : \bA_j \bx = \bb_j, \bx \in C_j \}$ is non-empty and satisfies a Slater condition for each $j$.
\end{assumption*}
This is not a particularly stringent condition; given data that does not satisfy it, we can always pre-process the data ensure it does satisfy this assumption.  

We seek a function $\mathbf{f}$ such that $\bx_j$ is an approximate solution to $\VI( \mathbf{f}, \bA_j, \bb_j, C_j)$ for each $j$.  Note, the function $\mathbf{f}$ is common to all observations.  Specifically, we would like to solve:
\begin{align}
\nonumber
\min_{\mathbf{f}, \bepsilon} \quad & \| \bepsilon \|
\\ \label{eq:Informal}
\text{s.t.} \quad & \bx_j \text{ is an $\epsilon_j$-approximate solution to } \VI(\mathbf{f}, \bA_j, \bb_j, C_j), \quad j = 1, \ldots, N, 
\\ \nonumber
& \mathbf{f} \in \mathcal{S}.
\end{align}
where $\| \cdot \|$ represents some choice of norm, and $\mathcal{S}$ represents the set of admissible functions.  In the parametric case, treated in the following section, we will assume that $\mathcal{S}$ is indexed by a vector of parameters $\btheta \in \Theta \subseteq \R^M$.  In the nonparametric case, $\mathcal{S}$ will be a general set of functions that satisfy certain smoothness properties.  We defer this extension until Section~\ref{Sec:NonParametric}.

%%%%%%%%%%%%%%%%%%%%%%%%%%%
\subsection{Parametric Estimation:}
\label{Sec:Parametric}
In this section, we assume that the function $\mathbf{f}$ is known to belong to a parametric family indexed by a vector $\btheta \in \Theta$.  We write $\mathbf{f}( \bx ; \btheta)$ to denote this dependence.  We will assume throughout that $\Theta$ is compact and $\mathbf{f}( \bx ; \btheta)$ is continuous in $\btheta$.  
A direct application of Theorem~\ref{thm:Aghassi} yields the following reformulation:
\begin{theorem} \label{thm:ParametricFormulation}
Under assumptions \ref{assumption1}, \ref{assumption2} and the additional constraint that $\mathbf{f} = \mathbf{f}(\bx; \btheta)$ for some $\btheta \in \Theta$, problem Eq.~\eqref{eq:Informal} can be reformulated as
\begin{align}
\label{eq:ParametricInvVI}
\min_{\btheta \in \Theta, \by, \bepsilon} \quad & \| \bepsilon \|
\\ \nonumber
\text{s.t.} \quad & \bA_j^T \by_j \leq_C \mathbf{f}( \bx_j; \btheta), \quad j = 1, \ldots N, 
\\ \nonumber
&\mathbf{f}( \bx_j; \btheta)^T \bx_j - \bb_j^T \by_j \leq \epsilon_j, \quad j = 1, \ldots, N,
\end{align}
where $\by = ( \by_1, \ldots, \by_N)$. 
\end{theorem}
\blue{\begin{remark}[Multiple equilibria]
We stress that since Theorem~\ref{thm:Aghassi} is true for any $\epsilon$-approximate solution to the VI, Theorem~\ref{thm:ParametricFormulation} is valid even when the function $\mathbf{f}$ might give rise to multiple distinct equilibria.  This robustness to multiple equilibria is an important strength of our approach that distinguishes it from other specialized approaches that require uniqueness of the equilibrium.
\end{remark} }
\blue{\begin{remark}[Equilibria on the boundary]
In Theorem~\ref{thm:Aghassi}, we did not need to assume that the $\bx_j$ or the solutions to $VI(\mathbf{f}, \bA_j, \bb_j)$ belonged to the interior of $\F_j$.  Consequently, Theorem~\ref{thm:ParametricFormulation} is valid even if the observations $\bx_j$ or induced solutions to $VI(\mathbf{f}, \bA_j, \bb_j)$ occur on the boundary.  This is in contrast to many other techniques which require that the solutions occur on the relative interior of the feasible set.  
\end{remark}}
\begin{remark}[Computational complexity]
Observe that $\bx_j$ are data in Problem~\eqref{eq:ParametricInvVI}, not decision variables.  Consequently, the complexity of this optimization depends on the cone $C$ and the dependence of $\mathbf{f}$ on $\btheta$, but \emph{not} on the dependence of $\mathbf{f}$ on $\bx$.  For a number of interesting parametric forms, we can show that Problem~\eqref{eq:ParametricInvVI} is in fact tractable.  
\end{remark}
%This gives rise to the following corollary.
%\begin{corollary}
%Under the assumptions of Theorem~\ref{thm:ParametricFormulation}, if the constraint $\btheta \in \Theta$ is conic representable and $\mathbf{f}(\cdot, \btheta)$ depends linearly on $\btheta$, then Problem~\ref{eq:ParametricInvVI} can be solved as a conic optimization problem.
%\end{corollary}
%For many standard cones, conic optimization problems are well-known to be tractable both theoretically and practically.  See \cite{ben2001lectures} for an overview of complexity results of conic optimization.  

As an example, suppose $\mathbf{f}(\bx; \btheta) = \sum_{i=1}^M \theta_i \bphi_i(\bx)$ where $\bphi_1( \bx ), \dots, \bphi_M( \bx )$ is a set of (nonlinear) basis functions.  Since $\mathbf{f}$ depends linearly on $\btheta$, Problem~\eqref{eq:ParametricInvVI} is a conic optimization problem, even though the basis functions $\phi_i(\bx)$ may be arbitrary nonlinear functions.  Indeed, if $C$ is the nonnegative orthant, Problem~\eqref{eq:ParametricInvVI} is a linear optimization problem.  Similarly, if $C$ is the second-order cone, Problem~\eqref{eq:ParametricInvVI} is a second-order cone problem.  

\blue{Finally, although structural estimation is not the focus of our paper, in Appendix~\ref{sec:Casting} 
we briefly illustrate 
how  to use Theorem \ref{thm:Aghassi} to formulate an alternate optimization problem that is similar to, but different from, Problem~\eqref{eq:ParametricInvVI} and 
closer in spirit to structural estimation techniques.  Moreover, we show that this formulation is equivalent to certain structural estimation techniques in the sense that they produce the same estimators.  This section may prove useful to readers interested in comparing these methodology.  }

%%%%%%%%%%%%%%%%%%%%%%%%%%%
\subsection{Application: Demand Estimation under Bertrand-Nash Competition}
\label{sec:ParametricDemand}
In this section, we use Theorem~\ref{thm:ParametricFormulation} to estimate an unknown demand function for a product so that observed prices are approximately in Bertrand-Nash equilibrium.  \blue{This is a somewhat stylized example inspired by various influential works in the econometrics literature, such as \cite{berry1995automobile} and \cite{berry1994estimating}.  We include this styled example for two reasons:
1) To illustrate a simple problem where equilibria may occur on the boundary of the feasible region.}
2) To further clarify how the choice of parameterization of $\mathbf{f}(\cdot; \boldsymbol{\theta})$ affects the computational complexity of the estimation problem.

For simplicity, consider two firms competing by setting prices $p_1, p_2$, respectively.  \blue{Demand for firm $i$'s product, denoted $D_i(p_1, p_2, \xi)$, is a function of both prices, and other economic indicators, such as GDP, denoted by $\xi$.}  
Each firm sets prices to maximize its own revenues $U_i( p_1, p_2, \xi) = p_i D_i(p_1, p_2, \xi)$ subject to the constraint $0 \leq p_i \leq \overline{p}$.  The upper bound $\overline{p}$ might be interpreted as a government regulation as is frequent in some markets for public goods, like electricity.  
\blue{We assume a priori that each demand function belongs to some given parametric family indexed by $(\btheta_1, \btheta_2) \in \Theta$: $D_1(p_1, p_2, \xi; \btheta_1), D_2(p_1, p_2, \xi; \btheta_2)$.  We seek to estimate $\btheta_1, \btheta_2 \in \Theta$ so that the data $(p_1^j, p_2^j, \xi)$ for $j = 1, \ldots, N$ correspond approximately to Nash equilibria.  }

\blue{Both \cite{berry1995automobile} and \cite{berry1994estimating} assume that equilibrium prices do not occur on the boundary, i.e., that $p_i < \overline{p}$ since they leverage Eq.~\eqref{eq:NashBad} in their analysis.  These methods are, thus, not directly applicable.  }

By contrast, Theorem~\ref{thm:ParametricFormulation} directly applies yielding (after some arithmetic)
{ \small
\begin{align} \nonumber 
\min_{\substack{\by, \bepsilon \\ (\btheta_1, \btheta_2) \in \Theta}}& \ \  \| \bepsilon \|
\\ \label{eq:DemandDeriv}
\text{s.t.} &\ \   \by^j \geq \bzero,  \ \  j = 1, \ldots, N,
\\ \nonumber
&y_i^j \geq p_i^j \frac{\partial}{\partial p_i} D_i(p_1^j, p_2^j, \xi^j; \btheta_i) + D_i(p_1^j, p_2^j, \xi^j; \btheta_i), \ \  i = 1, 2, \ j = 1, \ldots, N, 
\\ \nonumber
&\sum_{i=1}^2 
\overline{p}^j y_i^j -
(p_i^j)^2 \frac{\partial}{\partial p_i} D_i(p_1^j, p_2^j, \xi^j; \btheta_i) - p_i^j D_i(p_1^j, p_2^j, \xi^j; \btheta_i)  \leq \epsilon_j, \ \  j = 1, \ldots, N.  
\end{align}
}
We stress that potentially more complex constraints on the feasible region can be incorporated just as easily.  

Next, recall that the complexity of the optimization problem \eqref{eq:DemandDeriv} depends on the parameterization of $D_i(p_1, p_2, \xi, \btheta_i)$.  For example, when demand is linear, 
\begin{equation}
\label{eq:LinearDemand}
D_i(p_1, p_2, \xi; \btheta_i) = \theta_{i0} + \theta_{i1} p_1 + \theta_{i2}p_2 + \theta_{i3}\xi
\end{equation}
problem \eqref{eq:DemandDeriv} reduces to the linear optimization problem: 
\begin{align} \nonumber 
\min_{\by, \bepsilon, (\btheta_1, \btheta_2) \in \Theta, \bd} \quad & \| \bepsilon \|
\\ \nonumber
\text{s.t.} \quad &  \by^j \geq \bzero, \quad j = 1, \ldots, N,
\\ \nonumber
& y_i^j \geq d^j_i + \theta_{ii} p_i^j, \quad \quad  \quad i = 1, 2, \quad j = 1, \ldots, N, 
\\ \label{eq:DemandLP}
& \overline{p}\sum_{i=1}^2 
 y_i^j - p^j_id^j_i - (p_i^j)^2 \theta_{ii} \leq \epsilon_j, \quad j = 1, \ldots, N,
\\ \nonumber
& d_i^j = \theta_{i0} + \theta_{i1} p^j_1 + \theta_{i2}p^j_2 + \theta_{i3}\xi^j, \quad \quad i = 1, 2,  \quad j = 1, \ldots, N.  
\end{align}
Alternatively, if we assume demand is given by the multinomial logit model \cite{gallego2006price}, 
$
D_i(p_1, p_2, \xi; \btheta) = \frac{ e^{ \theta_{i0} + \theta_{i1} p_i + \theta_{i3}\xi} } 
						{ e^{ \theta_{10} + \theta_{11} p_1+ \theta_{13}\xi} +  e^{ \theta_{20} + \theta_{21} p_2 + \theta_{23}\xi} + \theta_{00} }, 
$
the problem \eqref{eq:DemandDeriv} becomes 
\begin{align*} 
\min_{\by, \bepsilon, \btheta_1, \btheta_2, \mathbf{d_1}, \mathbf{d_2} } \quad & \| \bepsilon \|
\\
\text{s.t.} \quad &  \by^j \geq \bzero, \quad j = 1, \ldots, N,
\\
& y_i^j \geq p_i^j \theta_{i1}d_1^jd_2^j  + d_i^j, \quad i = 1, 2,
\\
&\sum_{i=1}^2 \overline{p}^j y_i^j + p_i^j d_i^j - (p_i^j)^2 \theta_{1i} d_i^j (1-d_i^j)\leq \epsilon_j
\\
&d_i^j  = \frac{ e^{ \theta_{0i} + \theta_{i1} p^j_i + \theta_{i3}\xi^j} } 
						{ e^{ \theta_{10} + \theta_{11} p^j_1 + \theta_{13}\xi^j } +  e^{ \theta_{20} + 
							\theta_{21} p_2^j + \theta_{23}\xi} + \theta_{00} }, \quad i = 1, 2, j = 1, \ldots N, 
\end{align*}
which is non-convex.  Non-convex optimization problems can be challenging numerically and may scale poorly.  

{\blockblue
Finally, we point out that although it more common in the econometrics literature to specify the demand functions $D_i$ directly as we have above, one could equivalently specify the marginal revenue functions 
\[
M_i(p_1, p_2, \xi; \btheta_i)  = p_i \partial_i D_i(p_1, p_2, \xi; \btheta_i) + D_i(p_1, p_2, \xi; \btheta_i )
\]
and then impute the demand function as necessary.  We adopt this equivalent approach later in Section~\ref{sec:DemandEstimation}.
}

%%%%%%%%%%%%%%%%%%%%%%%%%%%%%%%%
%%%%
\section{Kernel Methods: Background}
\label{Sec:Kernels}
Intuitively, our nonparametric approach in the next section seeks the ``smoothest" function $\mathbf{f}$ which make the observed data approximate equilibria, where the precise notion of smoothness is determined by the choice of kernel.  Kernel methods have been used extensively in machine learning, most recently for feature extraction in context of support-vector machines or principal component analysis.  Our use of kernels, however, more closely resembles their application in spline interpolation and regularization networks (\cite{wahba1990spline}, \cite{girosi1993priors}).  

\blue{Our goal in this section is to develop a sufficiently rich set of scalar valued functions over which we can tractably optimize using kernel methods.  Consequently, we first develop some background.  Our review is not comprehensive.  A more thorough treatment of kernel methods can be found in either \cite{smola1998learning}, \cite{trevor2001elements} or \cite{evgeniou2000regularization}.}

\blue{ Let $\F \subseteq \R^n$ denote some domain. Let $k: \F \times \F \rightarrow \R$ be a symmetric function.  We will say that $k$ is a kernel if $k$ is positive semidefinite over $\F$, i.e., if
\[
\sum_{i=1}^N \sum_{j=1}^N c_i c_j k(\bx_i, \bx_j) \geq 0 \text{ for any choice of } N \in \mathbb{N}, 
\ \
\bc \in \R^N, 
\ \
\bx_i \in \F.
\]
Examples of kernels over $\R^n$ include:
\begin{description}
\item[Linear:] $k(\bx, \by) \equiv \bx^T\by$,  
\item[Polynomial:] $k(\bx, \by) \equiv ( c + \bx^T \by )^d$ for some choice of $c \geq 0$ and $d \in \mathbb{N}$,  
\item[Gaussian:] $k(\bx, \by) \equiv \exp( -c \| \bx - \by \|^2 )$ for some choice of $c >0 $.  
\end{description}
}

\blue{Let $k_\bx( \cdot ) \equiv k( \bx, \cdot)$ denote the function of one variable obtained by fixing the first argument of $k$ to $\bx$ for any $\bx \in \F$.  Define $\H_0$ to be the vector space of scalar valued functions which are representable as finite linear combinations of elements $k_\bx$ for some $\bx \in \F$, i.e., 
\begin{equation}
\label{eq:SpanofH0}
\H_0 = \left\{ \sum_{j=1}^N \alpha_j k_{\bx_j} : \bx_j \in \F,  \ N \in \mathbb{N},  \ \alpha_j \in \R, \ j = 1, \ldots, N, N \in \mathbb{N} \right\}.  
\end{equation}
Observe that $k_\bx \in \H_0$ for all $\bx \in \F$, so that in a sense these elements form a basis of the space $\H_0$.  On the other hand, for a given $f \in \H_0$, its representation in terms of these elements $k_{\bx_j}$ for $\bx_j \in \F$ need not be unique.  In this sense, the elements $k_\bx$ are not like a basis.
}

\blue{For any $f, g \in \H_0$ such that
\begin{align} \label{eq:ChoiceofRepresentation}
f = \sum_{j=1}^N \alpha_j k_{\bx_j}, 
\ \
g = \sum_{i=1}^N \beta_i k_{\bx_i},
\ \ \balpha, \boldsymbol{\beta} \in \R^N
\end{align} 
we define a scalar product
\begin{align}\label{eq:DefineInnerProduct}
\langle f, g \rangle_{\H_0} \ = \  \sum_{i=1}^N\sum_{j=1}^N \alpha_i \beta_j \langle k_{\bx_i}, k_{\bx_j} \rangle_{\H_0}
 \ 	\equiv  \ \sum_{i=1}^N\sum_{j=1}^N \alpha_i \beta_j k(\bx_i, \bx_j).
\end{align}
Since the representation in \eqref{eq:ChoiceofRepresentation} is not unique, for this to be a valid definition one must prove that the right-hand side of the last equality is independent of the choice of representation. It is possible to do so.  See \cite{smola1998learning} for the details.  Finally, given this scalar-product, we define the norm 
$
\| f \|_{\H_0} \equiv \sqrt{\langle f, f \rangle}_{\H_0}.
$
}

\blue{In what follows, we will actually be interested in the closure of $\H_0$, i.e., 
\begin{equation}
	\label{eq:SpanofH}
	\H = \overline{ \H_0 }.  
\end{equation}
We extend the scalar product $\langle \cdot, \cdot \rangle_{\H_0}$ and norm $\| \cdot \|_{\H_0}$ to $\H$ by continuity.  (Again, see \cite{smola1998learning} for the details).  Working with $\H$ instead of $\H_0$ simplifies many results.\footnote{\blue{For the avoidance of doubt, the closure in \eqref{eq:SpanofH} is with respect to the norm $\| \cdot \|_{\H_0}$.}}
}

As an example, in the case of the linear and polynomial kernels, the space $\H$ is finite dimensional and corresponds to the space of linear functions and the space of polynomials of degree at most $d$, respectively.  In the case of the Gaussian kernel, the space $\H$ is infinite dimensional and is a subspace of all continuous functions.  

If $f \in \H_0$ admits a finite representation as in Eq.\eqref{eq:ChoiceofRepresentation}, note that from Eq. \eqref{eq:DefineInnerProduct} we have for all  $\bx \in \F$
\begin{align}
\label{eq:RepProp}
\langle k_\bx, f \rangle_\H \ = \  \sum_{j=1}^N \alpha_j k(\bx, \bx_j)
\ = \  f(\bx).
\end{align}
In fact, it can be shown that this property applies to all $f \in \H$ (\cite{luenberger1997optimization}).  This is the most fundamental property of $\H$ as it allows us to relate the scalar product of the space to function evaluation.  Eq. \eqref{eq:RepProp} is termed the \emph{reproducing property} and as a consequence, $\H$ is called a \emph{Reproducing Kernel Hilbert Space (RKHS).}  

At this point, it may appear that RKHS are very restrictive spaces of functions.  In fact, it can be shown that any Hilbert space of scalar-valued functions for which there exists 
 a $c \in \R$ such that for each $f \in \H$, 
$
| f( \bx ) | \leq c \| f \|_\H$ for all $\bx \in \F$
 is an RKHS (\cite{luenberger1997optimization}).  Thus, RKHS are fairly general.    Practically speaking, though, our three previous examples of kernels  --linear, polynomial, and Gaussian --are by far the most common in the literature.   

We conclude this section with a discussion about the norm $\| f \|_\H$.  We claim that in each of our previous examples, the norm $\| f\|_\H$ makes precise a different notion of ``smoothness" of the function $f$.  For example, it is not hard to see that if $f(\bx) = \bw^T \bx$, then under the linear kernel $\| f \|_\H = \| \bw \|$.  Thus, functions with small norm have small gradients and are ``smooth" in the sense that they do not change value rapidly in a small neighborhood.  

Similarly, it can be shown (see \cite{girosi1993priors}) that under the Gaussian kernel, 
\begin{equation}
\label{eq:GaussianNorm}
\| f \|_\H^2 = \frac{1}{(2 \pi)^n} \int | \tilde{f}(\omega) |^2 e^{\frac{\|\omega\|^2}{2c}} d \omega, 
\end{equation}  
where $\tilde{f}$ is the Fourier transformation of $f$.  Thus, functions with small norms do not have many high-frequency Fourier coefficients and are ``smooth" in the sense that they do not oscillate very quickly.  

The case of the polynomial kernel is somewhat more involved as there does not exist a \emph{simple} explicit expression for the norm (see \cite{girosi1993priors}).  However, it is easily confirmed numerically using Eq. \eqref{eq:DefineInnerProduct} that functions with small norms do not have large coefficients and do not have have high degree.  Consequently, they are ``smooth" in the sense that their derivatives do not change value rapidly in a small neighborhood.  

Although the above reasoning is somewhat heuristic, it is possible to make the intuition that the norm on an RKHS describes a notion of smoothness completely formal.  The theoretical details go beyond the scope of this paper (see \cite{girosi1993priors}).  
For our purposes, an intuitive appreciation that the $\H$-norm penalizes non-smooth functions and that the particular notion of smoothness is defined by the kernel will be sufficient for the remainder.

%%%%%%%%%%%%%%%%%%%%%%%%%%%
\section{The Inverse Variational Inequality Problem: A Nonparametric Approach}
\label{Sec:NonParametric}
\subsection{Kernel Based Formulation}
In this section, we develop a nonparametric approach to the inverse variational inequality problem.  
The principal difficulty in formulating a nonparametric equivalent to \eqref{eq:Informal} is that the problem is ill-posed. Specifically, if the set $S$ is sufficiently rich, we expect there to be many, potentially infinitely many, different functions $\mathbf{f}$ which all reconcile the data, and make each observation an exact equilibrium.  Intuitively, this multiplicity of solutions is similar to the case of interpolation where, given a small set of points, many different functions will interpolate between them exactly.  Which function, then, is the ``right" one?  

We propose to select the function $\mathbf{f}$ of minimal $\H$-norm among those that approximately reconcile the data.  This choice has several advantages.  First, as mentioned earlier, functions with small norm are ``smooth", where the precise definition of smoothness will be determined by the choice of kernel.  We feel that in many applications, assuming that the function defining a VI is smooth is very natural.  Second, as we shall prove, identifying the function $\mathbf{f}$ with minimal norm is computationally tractable, even when the RKHS $\H$ is infinite dimensional.  Finally, as we will show in Section~\ref{sec:Generalization}, functions with bounded $\H$-norm will have good generalization properties.  

Using Theorem~\ref{thm:Aghassi}, we reformulate Problem~\eqref{eq:Informal} as

\begin{subequations}\label{eq:NonParametricConstrained}
\begin{align} \notag
\min_{\mathbf{f}, \by, \bepsilon} \quad &  \sum_{i=1}^n  \| f_i \|^2_{\H}
\\ \label{eq:dualfeasible}
\text{s.t.} \quad & 
\bA_j^T \by_j \leq \mathbf{f}(\bx_j), \quad  \quad \quad \quad  j = 1, \ldots, N, 
\\  \label{eq:dualitygap}
&\bx_j^T \mathbf{f}(\bx_j) - \bb_j^T\by_j  \leq \epsilon_j,  \quad j = 1, \ldots, N,
\\  \notag
& \| \bepsilon \| \leq \kappa, \quad \bepsilon \geq \bzero,  \quad 
f_i \in \H, \quad i = 1, \ldots, n, 
\\  \label{eq:normalization}
&\frac{1}{N}\sum_{j=1}^N \bx_j^T \mathbf{f}(\bx_j) = 1.
\end{align}
\end{subequations}

Here $f_i$ is the $i$-th component of the vector function $\mathbf{f}$ and $\H$ is an RKHS.   Since we may always scale the function $\mathbf{f}$ in $\VI(\mathbf{f}, \F)$ by a positive constant without affecting the solution, we require the last constraint as a normalization condition.  
Finally, the exogenous parameter $\kappa$ allows us to balance the norm of $\mathbf{f}$ against how closely $\mathbf{f}$ reconciles the data; decreasing $\kappa$ will make the observed data closer to equilibria at the price of $\mathbf{f}$ having greater norm.  

Problem~\eqref{eq:NonParametricConstrained} is an optimization over functions, and it is not obvious how to solve it.   We show in the next theorem, however, that this can be done in a tractable way.  This theorem is an extension of a representation theorem from the kernel literature (see \cite{wahba1990spline}) to the constrained multivariate case.  See the appendix for a proof.

\begin{theorem} \label{Prop:Finite}
Suppose Problem~\eqref{eq:NonParametricConstrained} is feasible.  Then, there exists an optimal solution $\mathbf{f}^* = (f^*_1, \ldots, f^*_n)$ with the following form:
\begin{equation}
\label{eq:KernelExpansion}
f^*_i = \sum_{j=1}^N \alpha_{i,j} k_{\bx_j},
\end{equation}
for some $\alpha_{i,j} \in \R$, where $k$ denotes the kernel of $\H$.  
\end{theorem}

By definition of $\H$, when Problem \eqref{eq:NonParametricConstrained} is feasible, its solution is a
potentially infinite expansion in terms of the kernel function evaluated at various points of $\F$.  The importance of Theorem~\ref{Prop:Finite} is that it allows us to conclude, first, that this expansion is in fact finite, and second, that the relevant points of evaluation are exactly the data points $\bx_j$.  This fact further allows us to replace the optimization problem \eqref{eq:NonParametricConstrained}, which is over an infinite dimensional space, with an optimization problem over a finite dimensional space.

\begin{theorem} \label{thm:QP}
\blue{Problem \eqref{eq:NonParametricConstrained} is feasible if and only if the following optimization problem is feasible:}
\begin{align} \nonumber
\min_{\balpha, \by, \bepsilon} \quad & \sum_{i=1}^n \be_i^T\balpha \bK  \balpha^T \be_i
\\ \nonumber
\text{s.t.}  \quad & 
\bA_j \by_j \leq \balpha \bK \be_j \quad j = 1, \ldots, N,
\\ \nonumber
& \bx_j^T \balpha \bK \be_j - \bb_j^T \by_j \leq \epsilon_j \quad j = 1, \ldots, N,
\\ \label{eq:QP} 
& \| \bepsilon\|  \leq \kappa, \quad \bepsilon \geq \bzero,  
\\ \nonumber
& \frac{1}{N}\sum_{j=1}^N \bx_j^T \balpha \bK \be_j = 1.
\end{align}
Here $\balpha = (\alpha_{ij})_{i=1, j=1}^{i=n, j = N} \in \R^{n \times N}$, and $\bK = (k( \bx_i, \bx_j))_{i, j=1}^{i, j=N}$.  
\blue{Moreover, given an optimal solution $\balpha$ to the above optimization problem, an optimal solution to Problem~\eqref{eq:NonParametricConstrained} is given by Eq.~\eqref{eq:KernelExpansion}.}
\end{theorem}
See the appendix for a proof.  
Given the optimal parameters $\balpha, \mathbf{f}$ can be evaluated at new points $\bt$ using \eqref{eq:KernelExpansion}. 
Note that $\bK$ is positive semidefinite (as a matrix) since $k$ is positive definite (as a function).  Thus, \eqref{eq:QP} is a convex, quadratic optimization problem.  Such optimization problems are very tractable numerically and theoretically, even for large-scale instances.  (See \cite{boyd2004convex}).  Moreover, this quadratic optimization problem exhibits block structure -- only the variables $\balpha_j$ couple the subproblems defined by the $\by_j$ --- which can be further exploited in large-scale instances.  Finally, the size of this optimization scales with $N$, the number of observations, not with the dimension of the original space $\mathcal{H}$, which may be infinite.

\blue{Observe that Problem~\eqref{eq:NonParametricConstrained} is bounded, but may be infeasible.  
%If it is infeasible, Problem~\eqref{eq:QP} will also be infeasible, since any solution to \eqref{eq:QP} can be transformed to a solution to \eqref{eq:NonParametricConstrained} using \eqref{eq:KernelExpansion}.  
We claim it will be feasible whenever $\kappa$ is sufficiently large.  Indeed, let $\mathbf{\hat{f}}_i \in \H$ be any functions from the RKHS.  By scaling, we can always ensure \eqref{eq:normalization} is satisfied.  The following convex optimization
$\min_{\bx : \bA_j \bx = \bb_j, \bx \geq \bzero } \mathbf{\hat{f}}(\bx_j)^T \bx $ is bounded and satisfies a Slater condition by Assumption \assref{assumption2}.  Let $\hat{\by}_j$ be the dual variables to this optimization, so that $\hat{\by}_j$ satisfy \eqref{eq:dualfeasible} and define $\hat{\epsilon}_j$ according to \eqref{eq:dualitygap}.  Then as long as $\kappa \geq \| \hat{\bepsilon} \|$, Problem~\eqref{eq:NonParametricConstrained}, and consequently Problem~\eqref{eq:QP}, will be feasible and obtain an optimal solution. }

\blue{Computationally, treating the possible infeasibility of \eqref{eq:QP} can be cumbersome, so in what follows, we find it more convenient to dualize this constraint 
so that the objective becomes, 
\begin{equation} \label{eq:DualizedObjective}
\min_{\balpha, \by}  \balpha^T K \balpha + \lambda 
\| \bepsilon \|,
\end{equation}
and then solve this problem for various choices of $\lambda>0$.  Note this version of the problem is always feasible, and, indeed, we will employ this formulation later in Sec. \ref{sec:Numerics}.  }

We conclude this section by contrasting our parametric and nonparametric formulations.  
Unlike the parametric approach, the nonparametric approach is always a convex optimization problem.  This highlights a key tradeoff in the two approaches.  The parametric approach offers us fine-grained control over the specific form of the function $\mathbf{f}$ at the potential expense of the tractability of the optimization.  The nonparametric approach offers less control but is more tractable.  

We next illustrate our nonparametric approach below with an example.  

%%%%%%%%%%%%%%%%%%%%%%%%%%%
\subsection{Application: Estimating the Cost Function in Wardrop Equilibrium.}
\label{Sec:ExampleTraffic}
Recall the example of Wardrop equilibrium from Section~\ref{Sec:Examples}.  In practice, while the network $(\mathcal{V}, \A)$ is readily observable, the demands $\bd^w$ and cost function $c_a( \cdot )$ must be estimated.  Although several techniques already exist for estimating the demands $\bd^w$ (\cite{Abrahamsson}, \cite{yang1992estimation}), there are fewer approaches for estimating $c_a(\cdot)$.  Those techniques that do exist often use stylized networks, e.g., one origin-destination pair, to build insights.  See \cite{Nakayama2007MethodEstimating} for a maximum likelihood approach, and \cite{perakis2006analytical} for kinematic wave analyses.

By contrast, we focus on estimating $c_a( \cdot)$ from observed flows or traffic counts on real, large scale networks.  \blue{Specifically, we assume we are given networks $(\mathcal{V}_j, \A_j)$, $j=1, \ldots, N$,  and have access to estimated demands on these networks $\bd^{\bw_j}$ for all $\bw_j \in W_j$.  In practice, this may be the same network observed at different times of day, or different times of year, causing each observation to have different demands.}  

In the transportation literature, one typically assumes that $c_a ( \cdot)$ only depends on arc $a$, and in fact, can be written in the form 
$\label{eq:FormofCongestion}
c_a( x_a) = c_{0a}g\left(\frac{x_a}{m_a}\right),
$
for some nondecreasing function $g$.  The constant $c_{0a}$ is sometimes called the free-flow travel time of the arc, and $m_a$ is the effective capacity of the arc.  These constants are computed from particular characteristics of the arc, such as its length, the number of lanes or the posted speed limit.  (Note the capacity $m_a$ is not a hard constraint; it not unusual to see arcs where $x^*_a > m_a$ in equilibrium.)    We will also assume this form for the cost function, and seek to estimate the function $g ( \cdot )$.  

Using \eqref{eq:QP} and \eqref{eq:DualizedObjective} we obtain the quadratic optimization problem
\begin{align} \label{eq:TrafficQP}
\min_{\balpha, \by, \bepsilon} \quad & \balpha^T  \bK \balpha + \lambda \| \bepsilon \|
\\ \nonumber
\text{s.t.} \quad &\be_a^T \bN_j^T \by^{\bw} \leq c_{0a} \balpha^T \bK \be_{a}, 
\quad  \forall \bw \in W_j, \ a \in \A_j, \ j=1, \ldots, N,  
\\ \label{eq:NonDecreasing}
& \balpha^T \bK \be_a \leq \balpha^T \bK \be_{a^\prime}, \quad \quad \quad \forall a, a^\prime \in \A_0 \quad \text{s.t.} \   \frac{x_a}{m_a} \leq \frac{x_{a^\prime}}{m_{a^\prime}},
\\ \nonumber
&\sum_{a \in \A_j} c_{0a} x_a \balpha^T \bK \be_a -
\sum_{\bw \in W_j} (\bd^{\bw})^T \by^{\bw} \leq \epsilon_j, \quad, \forall \bw \in W_j,  \ j=1, \ldots, N, 
\\ \nonumber
&
\balpha^T \bK \be_{a_0} = 1.
\end{align}
\blue{In the above formulation $\A_0$ is a subset of $\bigcup_{j=1}^N \A_j$ and $\bK \in \R^{\sum_{j=1}^N | \A_j |  \times \sum_{j=1}^N | \A_j | }$. Constraint \eqref{eq:NonDecreasing} enforces that the function $g(\cdot)$ be non-decreasing on these arcs.  Finally, $a_0$ is some (arbitrary) arc chosen to normalize the function. } 

\blue{Notice, the above optimization can be quite large.  If the various networks are of similar size, the problem has $O(N (| \A_1 | + | W_1 | | \mathcal{V}_1 |)$ variables and $O( N |W_1| |\A_1| + |\A_0| )$ constraints.  As mentioned previously, however, this optimization exhibits significant structure.  First, for many choices of kernel, the matrix $K$ is typically (approximately) low-rank.  Thus, it is usually possible to reformulate the optimization in a much lower dimensional space.  At the same time, for a fixed value of $\alpha$, the optimization decouples by $\bw \in W_j$ and $j$.  Each of these subproblems, in turn, is a shortest path problem which can be solved very efficiently, even for large-scale networks.  Thus, combining an appropriate transformation of variables with block decomposition, we can solve fairly large instances of this problem.  We take this approach in Section~\ref{sec:TrafficEstimation}. }

%%%%%%%%%%%%%%%%%%%%%%%%%%%
\section{Extensions}
\label{sec:extensions}
Before proceeding, we note that Theorem~\ref{Prop:Finite} actually holds in a more general setting.  Specifically, a minimization over an RKHS will admit a solution of the form \eqref{eq:KernelExpansion} whenever
\begin{enumerate}[label=\alph*)]
\item the optimization only depends on the norms of the components $\| f_i \|_\H$ and the function evaluated at a finite set of points $\mathbf{f}( \bx_j)$, and
\item the objective is nondecreasing in the norms $\| f_i \|_\H$.
\end{enumerate}
The proof is identical to the one presented above, and we omit it for conciseness.  An important consequence is that we can leverage the finite representation of Theorem~\ref{Prop:Finite} in a number of other estimation problems and to facilitate inference.  In this section, we describe some of these extensions.   

\subsection{Incorporating Priors and Semi-Parametric Estimation}
Suppose we believe a priori that the function $\mathbf{f}$ describing the VI should be close to a particular function $\mathbf{f}_0$ (a prior).  In other words, $\mathbf{f} = \mathbf{f}_0 + \mathbf{g}$ for some function $\mathbf{g}$ which we believe is small.  We might then solve
\begin{align} \nonumber 
\min_{\mathbf{g}, \by, \bepsilon} \quad &  \sum_{i=1}^n  \| g_i \|^2_{\H}
\\ \nonumber
\text{s.t.} \quad & 
\bA_j^T \by_j \leq \mathbf{f}_0(\bx_j) + \mathbf{g}(\bx_j) \quad  \quad \quad \quad  j = 1, \ldots, N, 
\\ \nonumber
&\bx_j^T (\mathbf{f}_0(\bx_j) + \mathbf{g}(\bx_j))  - \bb_j^T\by_j  \leq \epsilon_j  \quad j = 1, \ldots, N,
\\ \nonumber
& \| \bepsilon \| \leq \kappa, \quad \bepsilon \geq \bzero, 
\ \ 
g_i \in \H, \quad \quad \quad i = 1, \ldots, n. 
\end{align}
From our previous remarks, it follows that this optimization is equivalent to
\begin{align} \nonumber
\min_{\balpha, \by, \bepsilon} \quad & \sum_{i=1}^n \be_i^T\balpha \bK  \balpha^T \be_i
\\ \nonumber
\text{s.t.}  \quad & 
\bA_j \by_j \leq \mathbf{f}_0 ( \bx_j) + \balpha \bK \be_j \quad j = 1, \ldots, N
\\ \nonumber
& \bx_j^T (\mathbf{f}_0(\bx_j) + \balpha \bK \be_j) - \bb_j^T \by_j \leq \epsilon_j \quad j = 1, \ldots, N,
\\ \nonumber
& \| \bepsilon \| \leq \kappa, \quad \bepsilon \geq \bzero,  
\end{align}
which is still a convex quadratic optimization problem. 

In a similar way we can handle semi-parametric variants where $\mathbf{f}$ decomposes into the sum of two functions, one of which is known to belong to a parametric family and the other of which is defined nonparmetrically, i.e., $\mathbf{f} ( \cdot ) = \mathbf{f}_0( \cdot ; \btheta) + \mathbf{g}$ for some $\btheta$ and $\mathbf{g} \in \H^n$.  

\begin{remark}[A Challenge with Partial Derivatives]
There are natural modeling circumstances where Theorem~\ref{Prop:Finite} is not applicable.  For example, recall in our demand estimation example from Section~\ref{sec:ParametricDemand} that the inverse variational inequality problem depends not only on the demand functions $D_1( \cdot ), D_2( \cdot )$ evaluated at a finite set of points $(p_1^j, p_2^j)$, but also on their partial derivatives at those points.  Intuitively, the partial derivative $\partial_i D_i ( p_1^j, p_2^j)$ requires information about the function in a small neighborhood of $(p_1^j, p_2^j)$, not just at the point, itself.  Consequently, Theorem~\ref{Prop:Finite} is not applicable.  Extending the above techniques to this case remains an open area of research.  
\end{remark}

{ \blockblue
\subsection{Ambiguity sets}
\label{sec:ambiguity}
In many applications, there may be multiple distinct models which all reconcile the data equally well.  Breiman termed this phenomenon the ``Rashomon" effect.  It can occur even with parametric models that are well-identified, since there may exist models outside the parametric family which will also reconcile the data. 
Consequently, we would often like to identify the range of functions which may explain our data, and how much they differ.

We can determine this range by computing the upper and lower envelopes of the set of all functions within an RKHS that make the observed data approximate equilibria.  We call this set the ambiguity set for the estimator.  To construct these upper and lower bounds on the ambiguity set, consider fixing the value of $\kappa$ in \eqref{eq:NonParametricConstrained} and replacing the objective by $f_i( \bxhat)$ for some $\bxhat \in \F$.  This optimization problem satisfies the two conditions listed at the beginning of this section.  Consequently, Theorem~\ref{Prop:Finite} applies, and we can use the finite representation to rewrite the optimization problem as a linear optimization problem in $\balpha, \by$.  Using software for linear optimization, it is possible to generate lower and upper bounds on the function $\mathbf{f}(\bxhat)$ for various choices of $\bxhat$ quickly and efficiently.  

To what value should we set the constant $\kappa$?  One possibility is to let $\kappa$ be the optimal objective value of \eqref{eq:ParametricInvVI} , or a small multiple of it.  This choice of $\kappa$ yields the set of functions which ``best" reconcile the given data.  We discuss an alternative approach in Section~\ref{sec:Generalization} that yields a set of functions which are statistically similar to the current estimator.    

Regardless of how we choose, $\kappa$, though, ambiguity sets can be combined with our previous parametric formulations to assess the appropriateness of the particular choice of parametric family.  Indeed, the ambiguity set formed from the nonparametric kernel contains a set of alternatives to our parametric form which are, in some sense, equally plausible from the data.  If these alternatives have significantly different behavior from our parametric choice, we should exercise caution when interpreting the fitted function.  

Can we ever resolve the Rashomon effect?  In some cases, we can use application-specific knowledge to identify a unique choice.  In other cases, we need appeal to some extra, a priori criterion.  A typical approach in machine learning is to focus on a choice with good generalizability properties.  In the next section, we show that our proposed estimators enjoy such properties.  } 

%%%%%%%%%%%%%%%%%%%%%%%%%%%
{
\blockblue
\section{Generalization Guarantees}
\label{sec:Generalization}
In this section, we seek to prove generalization guarantees on the estimators from Problem~\eqref{eq:ParametricInvVI} and \eqref{eq:NonParametricConstrained}.  Proving various types of generalization guarantees for various algorithms is a central problem in machine learning.  These guarantees ensure that the performance of our estimator on new, future data will be similar to its observed performance on existing data. 

We impose a mild assumption on the generating process which is common throughout the machine learning literature:
\begin{assumption} \label{assumption:dataIID}
	The data $(\bx_j, \bA_j, \bb_j, C_j)$ are i.i.d. realizations of random variables $(\tilde{\bx}, \tilde{\bA}, \tilde{\bb}, \tilde{C})$ drawn from some probability measure $\P$.  
\end{assumption}
Notice, we make no assumptions on potential dependence between $(\tilde{\bx}, \tilde{\bA}, \tilde{\bb}, \tilde{C})$, nor do we need to know the precise form of $\P$.  We also assume
\begin{assumption} \label{ass:randomSlater} 
	The random set $\mathcal{\tilde{F}} = \{ \bx : \tilde{\bA} \bx = \tilde{\bb}, \bx \in \tilde{C} \}$ satisfies a Slater Condition almost surely.
\end{assumption}
\begin{assumption} \label{ass:Membership}
	$\tilde{\bx} \in \mathcal{\tilde{F}}$ almost surely.  
\end{assumption}
Assumptions \ref{ass:randomSlater} and \ref{ass:Membership} are not particularly stringent.  If these condition may fail, we can consider pre-processing the data so that they succeed, and then consider a new measure $\mathbb{Q}$ induced by this processing of $\P$.  

We now prove a bound for a special case of Problem~\eqref{eq:ParametricInvVI}.  Let $z_N, \btheta_N$ denote the optimal value and optimal solution of \eqref{eq:ParametricInvVI}.  If for some $N$, there exist multiple optimal solutions, choose $\btheta_N$ by some tie-breaking rule, e.g., the optimal solution with minimal $\ell_2$-norm.  For any $0< \alpha < 1$, define
\[
\beta(\alpha) \equiv \sum_{i=0}^{\text{dim}(\btheta)} \binom{N}{i} \alpha^i (1-\alpha)^{N-i}.
\]
\begin{theorem}\label{thm:Campi}
Consider Problem~\eqref{eq:ParametricInvVI} where the norm $\| \cdot \| = \| \cdot \|_\infty$.  Suppose that this problem is convex in $\btheta$ and that Assumptions~\assref{assumption1}, \assref{assumption:dataIID}-\assref{ass:Membership} hold. 
%Finally, let $(\bx_{N+1}, \bA_{N+1}, \bb_{N+1}, C_{N+1})$ denote a new, independent draw from $\P$.  
Then, for any $0 < \alpha < 1$, with probability at least $1-\beta(\alpha)$ with respect to the sampling, 
\begin{align*}
\P\left( \tilde{\bx} \text{ is a $z_N$-approximate equilibrium for } \VI(\mathbf{f}(\cdot, \btheta_N), \tilde{\bA}, \tilde{\bb}, \tilde{C}) \right) \geq 1 -\alpha. 
\end{align*}
%Moreover, there exist instances where this bound is tight.  
\end{theorem}
The proof relies on relating Problem~\eqref{eq:ParametricInvVI} to an uncertain convex program \cite{campi2008exact}, and leveraging results on the randomized counterparts of such programs.  See the appendix for the details.

\begin{remark}
There are two probability measures in the theorem.  The first (explicit) is the probability measure of the new data point $(\tilde{\bx}, \tilde{\bA}, \tilde{\bb}, \tilde{C})$.  The second (implicit) is the probability measure of the random quantities $z_N$, $\btheta_N$.  One way to interpret the theorem is as follows:  One can ask, ``For a fixed pair $z_N, \btheta_N$, is the probability that $\bx_{N+1}$ is a $z_N$-approximate equilibrium for  $\VI(\mathbf{f}(\cdot, \btheta_N), \bA_{N+1}, \bb_{N+1}, C_{N+1})$ with respect to the first measure at least $1-\alpha$?"  The theorem asserts the answer is, ``Yes" with probability at least $1-\beta(\alpha)$ with respect to the second measure.  More loosely, the theorem asserts that for ``typical" values of $z_N, \btheta_N$, the answer is ``yes."  This type of generalization result, i.e., a result which is conditional on the data-sampling measure, is typical in machine learning.  
\end{remark}
\begin{remark}
Notice that $\beta(\alpha)$ corresponds to the tail probability of a binomial distribution, and, hence, converges exponentially fast in $N$.  
\end{remark}
\begin{remark}[$\ell_1$ Regularization] \label{rem:L1Reg}
The value $\beta(\alpha)$ depends strongly on the dimension of $\theta$.  In \cite{campi2013random}, the authors show that including an $\ell_1$ regularization of the form $\| \theta \|_1$  to reduce the effective dimension of $\theta$ can significantly improve the above bound in the context of uncertain convex programs.\footnote{In fact, the authors show more:  they give an algorithm leveraging $\ell_1$ regularization to reduce the dimensionality of $\theta$ and then an improved bound based on the reduced dimension.  The analysis of this improved bound can be adapted to our current context at the expense of more notation. We omit the details for space.}  Motivated by this idea, we propose modifying our original procedure by including a regularization $\lambda \| \btheta \|_1$ in the objective of Problem~\eqref{eq:ParametricInvVI}.  Since the problem is convex this formulation is equivalent to including a constraint of the form $ \| \btheta \|_1 \leq \kappa$ for some value of $\kappa$ that implicitly depends on $\lambda$, and, consequently,  Theorem~\ref{thm:Campi} still applies but with $z_N$ redefined to exclude the contribution of the regularization to the objective value.  
%We will show numerically in Section~\ref{sec:Numerics} that this additional regularization can improve the quality of fit.  
\end{remark}

Unfortunately, the proof of Theorem~\ref{thm:Campi} doesn't generalize easily to other problems, such as other norms or Problem~\eqref{eq:NonParametricConstrained}.  A more general approach to proving generalization bounds is based upon Rademacher complexity.  Rademacher complexity is a popular measure of the complexity of a class of functions, related to the perhaps better known VC-bounds.  Loosely speaking, for function classes with small Rademacher complexity, empirical averages of functions in the class converge to their true expectation uniformly over the class, and there exist bounds on the rate of convergence which are tight up to constant factors.  We refer the reader to \cite{bartlett2003rademacher} for a formal treatment.  

We will use bounds based upon the Rademacher complexity of an appropriate class of functions to prove generalization bounds for both our parametric and nonparametric approaches.  In the case of our nonparametric approach, however, it will prove easier to analyze the following optimization problem instead of Problem~\eqref{eq:NonParametricConstrained}:
\begin{align} \notag
\min_{\mathbf{f}, \by, \bepsilon} \quad &  \frac{ \| \bepsilon \|_p^p}{N} 
\\ \notag
\text{s.t.} \quad & 
\bA_j^T \by_j \leq \mathbf{f}(\bx_j), \quad  \quad \quad \quad  j = 1, \ldots, N, 
\\  \label{eq:NonParametricRad}
&\bx_j^T \mathbf{f}(\bx_j) - \bb_j^T\by_j  \leq \epsilon_j,  \quad j = 1, \ldots, N,
\\  \notag
& \| f_i \|^2_{\H} \leq \kappa_i, \quad
f_i \in \H, \quad i = 1, \ldots, n, 
\\  \notag
&\frac{1}{N}\sum_{j=1}^N \bx_j^T \mathbf{f}(\bx_j) = 1,
\end{align}
for some fixed $p$, $1 \leq p < \infty$.  We have made two alterations from Problem~\eqref{eq:NonParametricConstrained} with the dualized objective \eqref{eq:DualizedObjective}.  First, using Lagrangian duality, we have moved the term $\lambda \sum_{i=1}^n \| f_i \|_{\H}$ from the objective to the constraints.  Indeed, for any value of $\lambda$, there exists values $\kappa_i$ so that these two problems are equivalent.  Second, we have specialized the choice of norm to a $p$-norm, and then made an increasing transformation of the objective.  
If we can show that solutions to Problem~\eqref{eq:NonParametricRad} enjoy strong generalization guarantees, Problem~\eqref{eq:NonParametricConstrained} should satisfy similar guarantees.  Now, introduce

\begin{assumption}	
\label{ass:boundedFeas} 
	The set $\mathcal{\tilde{F}}$ is contained within a ball of radius $R$ almost surely.  
\end{assumption}

Next, we introduce some additional notation.  Consider Problem~\eqref{eq:ParametricInvVI}.  Define
\begin{equation}
\label{eq:ParametricB}
2 \sup_{\substack{ \bx: \| \bx \|_2 \leq R \\ \btheta \in \Theta}}  \| \mathbf{f}(\bx ; \btheta) \|_2 \equiv  \overline{B}.
\end{equation}
Observe $\overline{B} < \infty$.  Let $\mathbf{f}_N = \mathbf{f}(\cdot; \btheta_N)$ denote the function corresponding to the optimal solution of Problem~\eqref{eq:ParametricInvVI}.  With a slight abuse of language, we call $\mathbf{f}_N$ a solution to Problem~\eqref{eq:ParametricInvVI}.  

We define analogous quantities for Problem~\eqref{eq:NonParametricRad}.  Given a kernel $k(\cdot, \cdot)$, let $\overline{K}^2 \equiv \sup_{\bx : \| \bx \|_2 \leq R } k( \bx, \bx)$.  Notice, if $k$ is continuous, $\overline{K}$ is finite by \assref{ass:boundedFeas}.  For example, 
\begin{equation} \label{eq:Kbar}
\overline{K}^2
= \begin{cases}  R^2 & \text{ for the linear kernel }
				\\  (c + R^2 )^d & \text{ for the polynomial kernel }
				\\ 1 & \text{ for the Gaussian kernel }
	\end{cases} 
\end{equation}
With a slight abuse of notation, let $z_N, \mathbf{f}_N$ denote the optimal value and an optimal solution to Problem~\eqref{eq:NonParametricRad}, and let  $\overline{B} \equiv 2 R \overline{K} \sqrt{ \sum_{i=1}^n \kappa_i^2 }$.  This mild abuse of notation allows us to express our results in a unified manner.  It will be clear from context whether we are treating Problem~\eqref{eq:ParametricInvVI} or Problem~\eqref{eq:NonParametricRad}, and consequently be clear which definition of $\mathbf{f}_N, \overline{B}$ we mean.  

Finally, define $\epsilon(\mathbf{f}_N, \tilde{\bx}, \tilde{\bA}, \tilde{\bb}, \tilde{C})$  
 to be the smallest $\epsilon \geq 0$ such that $\tilde{\bx}$ is an $\epsilon$-approximate solution to $\VI(\mathbf{f}_N, \tilde{\bA}, \tilde{\bb}, \tilde{C})$.  

\begin{theorem} \label{thm:Rademacher}
Let $z_N, \mathbf{f}_N$ be the optimal objective and an optimal solution to Problem~\eqref{eq:ParametricInvVI} or \eqref{eq:NonParametricRad}.  
Assume \assref{assumption1}, \assref{assumption:dataIID}-\assref{ass:boundedFeas}.   For any $0 < \beta < 1$, with probability at least $1-\beta$ with respect to the sampling,
\begin{enumerate} [label=\roman*)]
\item 
\begin{equation} \label{eq:ExpectationBound}
\E[ (\epsilon(\mathbf{f}_N, \tilde{\bx}, \tilde{\bA}, \tilde{\bb}, \tilde{C}))^p ] \leq
z_N + \frac{1}{\sqrt{N}} \left(4 p \overline{B}^p + 2\overline{B}^{p/2} \sqrt{2 \log( 2/\beta ) } \right).
\end{equation}
\item For any $\alpha > 0$, 
\begin{align*}
\P( \tilde{\bx} \text{ is a } z_N + &\text{$\alpha$-approximate equilibrium for} \VI( \mathbf{f}_N, \tilde{\bA}, \tilde{\bb}, \tilde{C}) ) \\ 
&\geq 
1 - \frac{1}{\alpha^p\sqrt{N}}  \left(4 p \overline{B}^p + 2\overline{B}^{p/2} \sqrt{2 \log( 2/\beta ) } \right).
\end{align*}
\end{enumerate}
\end{theorem}
\begin{remark}
	To build some intuition, consider the case $p=1$.  The quantity $z_N$ is the average error on the data set for $\mathbf{f}_N$.  The theorem shows with high-probability, $\mathbf{f}_N$ will make a new data point an $\epsilon$-approximate equilibrium, where $\epsilon$ is only $O(1/\sqrt{N})$ larger than $z_N$.  In other words, the fitted function will perform not much worse than the average error on the old data.  Note, this does not guarantee that $z_N$ is small.  Indeed, $z_N$ will only be small if in fact a $\VI$ is a good model for the system.
\end{remark}
\begin{remark}[Specifying Ambiguity Sets] \label{rem:Ambiguity}
We can use Theorem~\ref{thm:Rademacher} to motivate an alternate proposal for specifying $\kappa$ in ambiguity sets as in Section~\ref{sec:extensions}.  Specifically, let $R_N$ denote the second term on the righthand side of \eqref{eq:ExpectationBound}.  Given another feasible function $\mathbf{f}^\prime$ in Problem~\eqref{eq:NonParametricRad} whose objective value is strictly greater than $z_N + R_N$, we can claim that with probability at least $1-\beta$, $\mathbf{f}_N$ has a smaller expected approximation error than $\mathbf{f}^\prime$.  However, if the objective value of $\mathbf{f}^\prime$ is smaller than $z_N + R_N$, we cannot reject it at level $1-\beta$; it is statistically as plausible as $\mathbf{f}_N$.  Setting $\kappa = R_N$ in our ambiguity set recovers all such ``statistically plausible" functions.
\end{remark}

Theorems~\ref{thm:Campi} and \ref{thm:Rademacher} provide a guarantee on the generalization error of our method.  We may also be interested in its predictive power.  Namely, given a new point  $(\bx_{N+1}, \bA_{N+1}, \bb_{N+1}, C_{N+1})$, let
$\hat{\bx}$  be a solution to $\VI(\mathbf{f}_N, \bA_{N+1}, \bb_{N+1}, C_{N+1}))$
The value $\hat{\bx}$ is a prediction of the state of a system described by $(\bA_{N+1}, \bb_{N+1}, C_{N+1})$ using our fitted function and $\bx_{N+1}$ represents true state of that system.  We have the following theorem:
\begin{theorem} \label{thm:Predictive} 
Assume $\mathbf{f}_N$ is strongly monotone with parameter $\gamma$.
\begin{enumerate}[label=\roman*)]
\item Suppose the conditions of Theorem~\ref{thm:Campi} hold.  Then, for any $0 < \alpha < 1$, with probability at least $1-\beta(\alpha)$ with respect to the sampling, 
\[
\| \bx_{N+1} - \hat{\bx} \| \leq \sqrt{\frac{z_N}{\gamma}}.
\]
\item Suppose the conditions of Theorem~\ref{thm:Rademacher} hold. Then, for any $0 < \beta < 1$, with probability at least $1-\beta$ with respect to the sampling, for any $\alpha > 0$
\[
\P\left( \| \bx_{N+1} - \hat{\bx} \| > \sqrt{\frac{z_N + \alpha}{\gamma}} \right) \leq 
\frac{1}{\alpha^p\sqrt{N}}  \left(4 p \overline{B}^p + 2\overline{B}^{p/2} \sqrt{2 \log( 2/\beta ) } \right).
\]
\end{enumerate}
\end{theorem}

In words, Theorem~\ref{thm:Predictive} asserts that solutions to our $\VI$ using our fitted function serve as good predictions to future data realizations.  This is an important strength of our approach as it allows us to predict future behavior of the system.  Again, this is contingent on the fact that $z_N$ is small, i.e., that the $\VI$ well-explains the current data.  

We conclude this section by noting that experimental evidence from machine learning suggests that bounds such as those above based on Rademacher complexity can be loose in small-samples.  The recommended remedy is that, when computationally feasible, to use a more numerically intensive method like cross-validation or bootstrapping to estimate approximation and prediction errors.  This approach applies equally well to choosing parameters like the threshold in an ambiguity set $\kappa$ as described in Remark~\ref{rem:Ambiguity}.  We employ both approaches in Section~\ref{sec:Numerics}.    
}
%%%%%%%%%%%%%%%%%%%%%%%%%%%
\section{Computational Experiments}
\label{sec:Numerics}
{ \blockblue

In this section, we provide some computational experiments illustrating our approach.  
For concreteness, we focus on our two previous examples: estimating the demand function in Bertrand-Nash equilibrium from Sec. \ref{sec:ParametricDemand} and estimating cost functions in traffic equilibrium from Sec. \ref{Sec:ExampleTraffic}.  

Before providing the details of the experiments, we summarize our major insights.  
\begin{enumerate}
\item In settings where there are potentially many distinct functions that explain the data equally well, our nonparametric ambiguity sets are able to identify this set of functions.  By contrast, parametric methods may misleadingly suggest there is only one possible function.  
\item Even in the presence of endogenous, correlated noise, our parametric and nonparametric techniques are able to learn functions with good generalizability, even if the specified class does not contain the true function generating the data.
\item Sometimes, the functions obtained by our method are not strongly monotone.  Nonetheless, they frequently still have reasonable predictive power.    
\end{enumerate}
}

{\blockblue
\subsection{Bertrand-Nash Equilibrium (Full-Information)}
%%VG Rewrite this as you get data, you posit an a parametric form, and you got lucky and it fit perfect.  Make sure to stress that no amount of data would ever convince you you were wrong.  
\label{sec:DemandEstimation}
We first consider an idealized, full-information setting to illustrate the importance of our ambiguity set technique.  Specifically, we assume the true, demand functions are given by the nonlinear model 
\[
D^*_i(p_1, p_2, \xi_i) = \log( p_i) + \theta^*_{i1}p_1 + \theta^*_{i2}p_2 + \theta^*_{i3} \xi_i + \theta^*_{i4}, \quad i = 1, 2
\]
with 
$
\btheta^*_1 = [ -1.2, .5, 1, -9]^T$  and 
$
\btheta^*_2 = [.3, -1, 1, -9]^T.
$
We assume (for now) that although we know the parametric form of these demand functions, we do not know the precise values of $\btheta_1, \btheta_2$ and seek to estimate them.  The corresponding marginal revenue functions are
\begin{equation}
\label{eq:MargRevTrue}
M^*_i(p_1, p_2, \xi_i; \btheta^*_i) = \log(p_i) + \theta^*_{i1}p_1 + \theta^*_{i2}p_2 + \theta^*_{i3} \xi_i + \theta^*_{i4} + 1 + \theta^*_{ii}p_i, \quad i =1, 2.
\end{equation}
Here $\xi_1, \xi_2$ are random variables representing firm-specific knowledge which change over time (``demand shocks") causing prices to shift.  

Our idealized assumption is that $\xi_1=\xi_2 \equiv \xi$, and $\xi$ is common knowledge to both the firms and to the researcher (full-information).  In our simulations, we take $\xi$ to be i.i.d normals with mean $5$ and standard deviation $1.5$.  Using these parameters with $\overline{p} = .45$, we simulate values of $\xi$ and solve for the equilibrium prices $p_1^j, p_2^j$ for $j = 1, \ldots, 250$.  The values $(\xi^j, p_1^j, p_2^j)$ constitute our data set.  

To estimate $\btheta_1, \btheta_2$, we substitute the functional form Eq.~\eqref{eq:MargRevTrue} into Problem~\eqref{eq:DemandDeriv}, adding additional constraints that 1) the marginal revenue of firm $i$ is positive for the minimal price $p_i^j$ observed in the data, 2) the marginal revenue of firm $i$ is decreasing in firm $i$'s price, and 3) a normalization constraint.  (See Appendix~\ref{app:DemandForm} for an explicit formulation).

Unsurprisingly, solving this optimization recovers the true marginal revenue functions exactly.  We say ``unsurprisingly" because with full-information a correctly specified, known parametric form, we believe any reasonable estimation procedure should recover the true marginal revenue functions.  We point out that the optimal solution to the optimization problem is \emph{unique}, and the optimal value of the residuals is $\bepsilon = 0$

We plot the true marginal revenue functions for each firm (which is the same as our fitted function) in Figure~\ref{fig:IdealFit} (dashed black line).  To graph these functions we fixed $\xi$ to be its median value over the dataset, and fixed the other firm's price to be the price observed for this median value.  For convenience in what follows, we term this type of fixing of the other variables, \emph{fixing to the median observation.}  
\begin{figure}
\centering
\begin{minipage}{0.45\textwidth}
		\includegraphics[width=\textwidth]{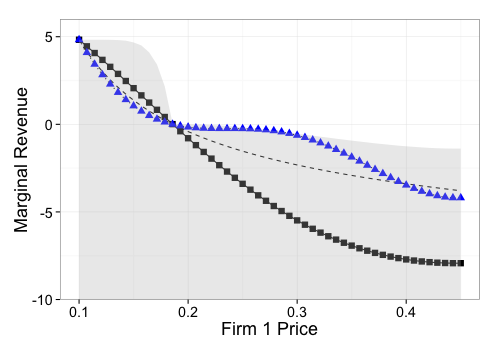}
\end{minipage}
\begin{minipage}{0.45\textwidth}
		\includegraphics[width=\textwidth]{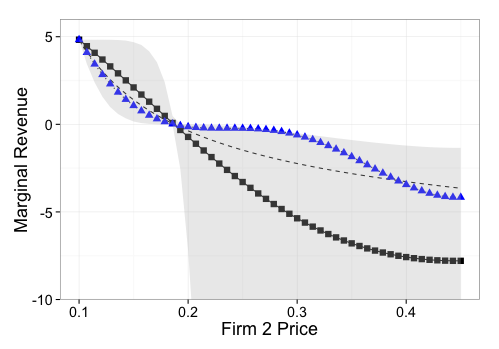}
\end{minipage}
\caption{\label{fig:IdealFit} An idealized scenario.  The true marginal revenue function (dashed black line), our nonparametric fit (black line, square markers), and the ambiguity set (grey region) for both firms.  Every function in the ambiguity set \emph{exactly} reconciles all the data.  A sample member (blue line, triangle markers) shown for comparison.  All variables other than the firm's own price have been fixed to the median observation.}
\end{figure}

Next consider the more realistic setting where we do not know the true parametric form \eqref{eq:MargRevTrue}, and so use our nonparametric method (cf. Problem~\ref{eq:NonParametricConstrained} with dualized objective \eqref{eq:DualizedObjective}).  We use a Gaussian kernel and tune the parameter $c$ and regularization constant $\lambda$ by 10-fold cross-validation.  The resulting fitted function is shown in Figure~\ref{fig:IdealFit} as a black line with square markers.  Notice, in particular, this function does not coincide with the true function.  However, this function also \emph{exactly} reconciles the data, i.e. the optimal value of the residuals is $\bepsilon = 0$.  This may seem surprising; the issue is that although there is only one function within the parametric family \eqref{eq:MargRevTrue} which reconciles the data, there are many potential smooth, nonparametric functions which also exactly reconcile this data.  Using our ambiguity set technique, we compute the upper and lower envelopes of this set of functions, and display the corresponding region as the grey ribbon in Figure~\ref{fig:IdealFit}.  We also plot a sample function from this set (blue line with triangle markers).  

This multiplicity phenomenon is not unusual; many inverse problems share it.  Moreover, it often persists even for very large samples $N$.  In this particular case, the crux of the issue is that, intuitively, the equilibrium conditions only give local information about the revenue function about its minimum.  (Notice all three marginal revenue functions cross zero at the same price).  The conditions themselves give no information about the global behavior of the function, even as $N\rightarrow \infty$.  

We see our ambiguity set technique and nonparametric analysis as important tools to protect against potentially faulty inference in these settings.  Indeed, parametric estimation might have incorrectly led us to believe that the unique marginal revenue function which recovered the data was the dashed line in Figure~\ref{fig:IdealFit} -- its residual error is zero and it is well-identified within the class.  We might then have been tempted to make claims about the slope of the marginal revenue function at the optima, or use it to impute a particular functional form for the demand function.  In reality, however, \textbf{\emph{any function from the ambiguity set might have just as easily generated this data}}, e.g., the blue line with triangle markers.  Those previous claims about the slope or demand function, then, need not hold.  The data does not support them.
Calculating nonparametric sets of plausible alternatives helps guard against these types of unwarranted claims.  

Finally, in the absence of any other information, we argue that our proposed nonparametric 
fit (red line with circles) is a reasonable candidate function in this space of alternatives.  By construction it will be smooth and well-behaved.  More generally, of all those functions which reconcile the data, it has the smallest $\H$-norm, and thus, by our generalization results in Section~\ref{sec:Generalization}, likely has the strongest generalization properties.  

\subsection{Bertrand-Nash Equilibrium (Unobserved Effects)}
\label{sec:DemandEstimation2}
We now proceed to a more realistic example.  Specifically, we no longer assume that $\xi_1 = \xi_2$, but rather these values represent (potentially different), firm-specific knowledge that is unobservable to us (the researchers).  We assume instead that we only have access to the noisy proxy $\xi$.  In our simulations, we take $\xi_1, \xi_2, \xi^\prime$ to be i.i.d normal with mean $5$ and standard deviation $1.5$, and let $\xi = (\xi_1 + \xi_2 + \xi^\prime)/3$.  Moreover, we assume that we we have incorrectly specified that the marginal revenue functions are of the form 
\begin{equation}\label{eq:ParametricGuess}
M_i(p_1, p_2, \xi; \btheta_i) = \sum_{k=1}^{9} \theta_{i1k} e^{ -k p_1} +  \sum_{k=1}^9 \theta_{i2k} e^{ -k  p_2} + 
							\theta_{i1} p_1 + \theta_{i2} p_2 + \theta_{i3} \xi_3 + \theta_{i4}
\end{equation}
for some values of $\btheta_1, \btheta_2$. Notice that the true parametric is not contained in this class.  This setup thus includes correlated noise, and endogenous effect, and parametric mispecification.  These features are known to cause statistical challenges in simple estimation procedures.  We simulate $N = 40$ observations ($\xi^j, p_1^j, p_2^j)$ from this model.  

We again fit this model first by solving a modification of Problem~\eqref{eq:DemandDeriv} as before.  (See Appendix~\ref{app:DemandForm} for an explicit formulation).  We only use half the day (20 observations) for reasons that will become clear momentarily.  We use the $\ell_\infty$-norm for the residuals $\bepsilon$ and an $\ell_1$-regularization of $\btheta_1, \btheta_2$ in the objective as discussed in Remark~\ref{rem:L1Reg}.  We tune the value of $\lambda$ in the regularization to minimize the mean squared error in price prediction obtaining the value $\lambda = .01$.  
%Specifically, recall that prediction corresponds to solving a VI with our fitted functions $M_i(p_1, p_2, \xi, \btheta_i)$ to obtain prices for both firms.  We estimate the prediction error in these prices via $10$-fold cross-validation on the data.  The left panel of Figure~\ref{fig:CrossValPriceDiffs} shows this mean square prediction error for varying $\lambda$ with $1$ standard deviation error bars.  Notice the value 
%$\lambda=.01$ has a lower cross-validated error than the unregularized fit ($\lambda = 0$).  

Unfortunately, because we used cross-validation to choose $\lambda$, Theorem~\ref{thm:Campi} does not directly apply.  Consequently, we now refit $\btheta_1, \btheta_2$ with $\lambda=.1$ using the other half of our training set.  
The fitted marginal revenue functions for $\lambda=.01$ can be seen in Figure ~\ref{fig:FittedDemands} (red line, circular markers).    
\begin{figure}
\begin{center}
\begin{minipage}{0.45\textwidth}
		\includegraphics[width=\textwidth]{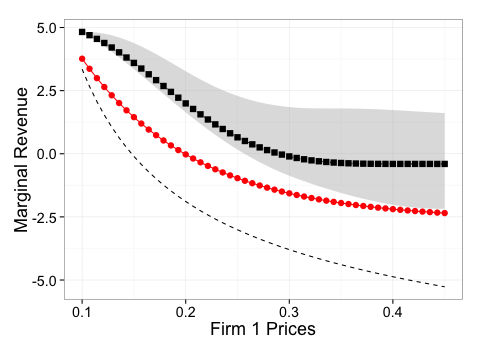}
\end{minipage}
\begin{minipage}{0.45\textwidth}
		\includegraphics[width=\textwidth]{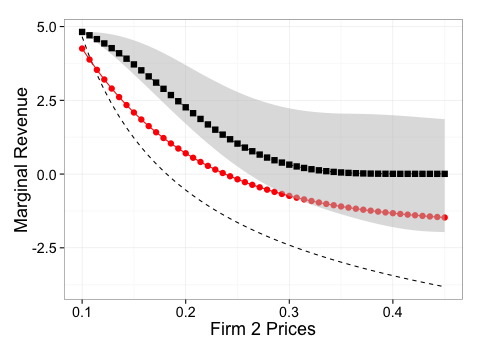}
\end{minipage}
\end{center}
\caption{\label{fig:FittedDemands} The true marginal revenue function (dashed line), fitted parametric marginal revenue function (solid red line, circular markers), fitted non-parametric marginal revenue function (solid black line, square  markers)and ambiguity sets (grey region) for each firm.
We fix all variables except the firm's own price to the median observation. }
\end{figure}
Notice that the fitted function does not exactly recover the original function, but does recover its approximate shape. 

To assess the out of sample performance of this model, we generate a new set of $N_{out}=200$ points.  For each point we compute the approximation error (minimal $\epsilon$ to make this point an $\epsilon$-approximate equilbria), and the prediction error had we attempted to predict this point by the solution to our $\VI$ with our fitted function.  Histograms of both quantities are in Fig.~\ref{fig:DemandHists}.
\begin{figure}
\begin{center}
\begin{minipage}{0.45\textwidth}
		\includegraphics[width=\textwidth]{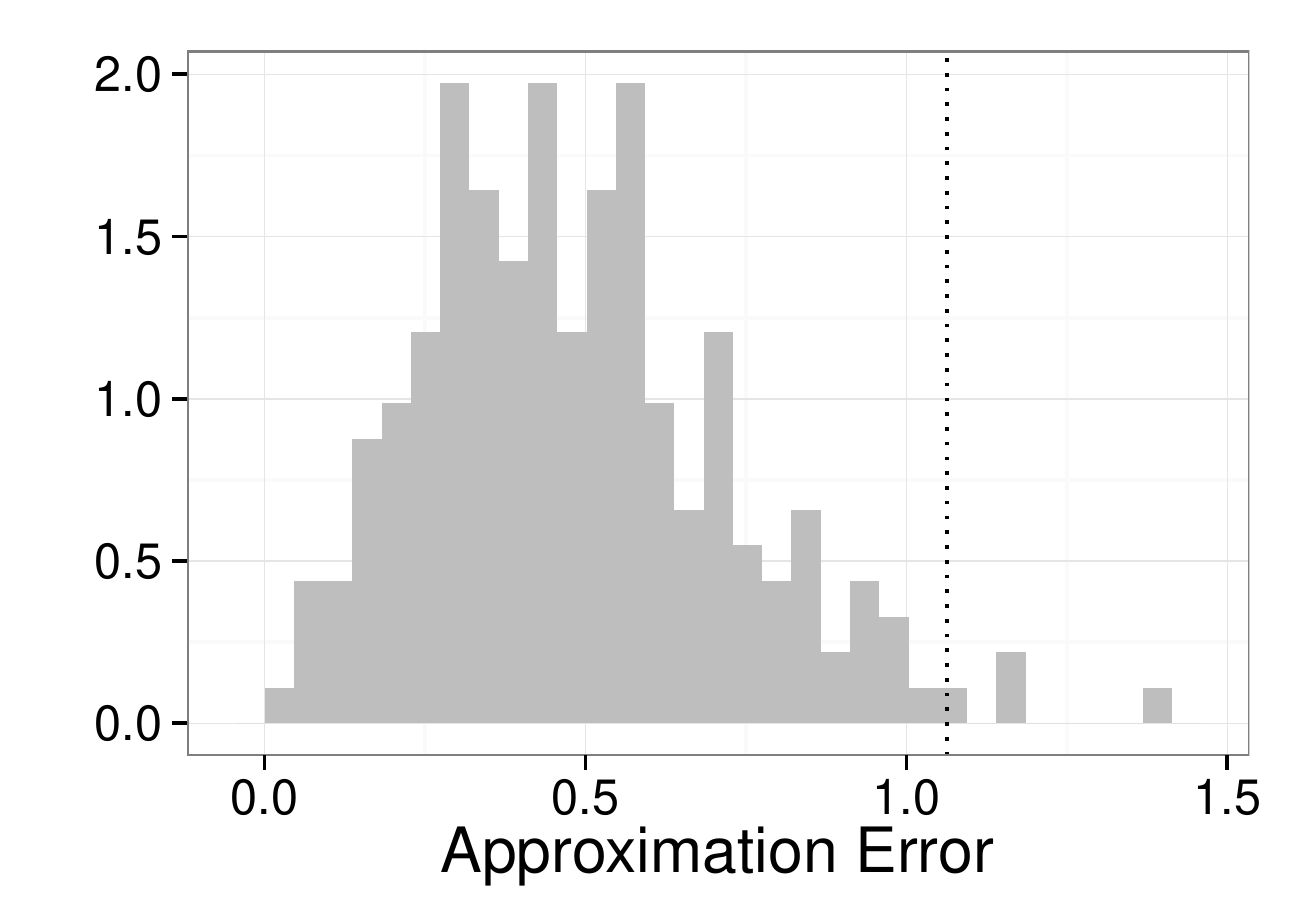}
\end{minipage}
\begin{minipage}{0.45\textwidth}
		\includegraphics[width=\textwidth]{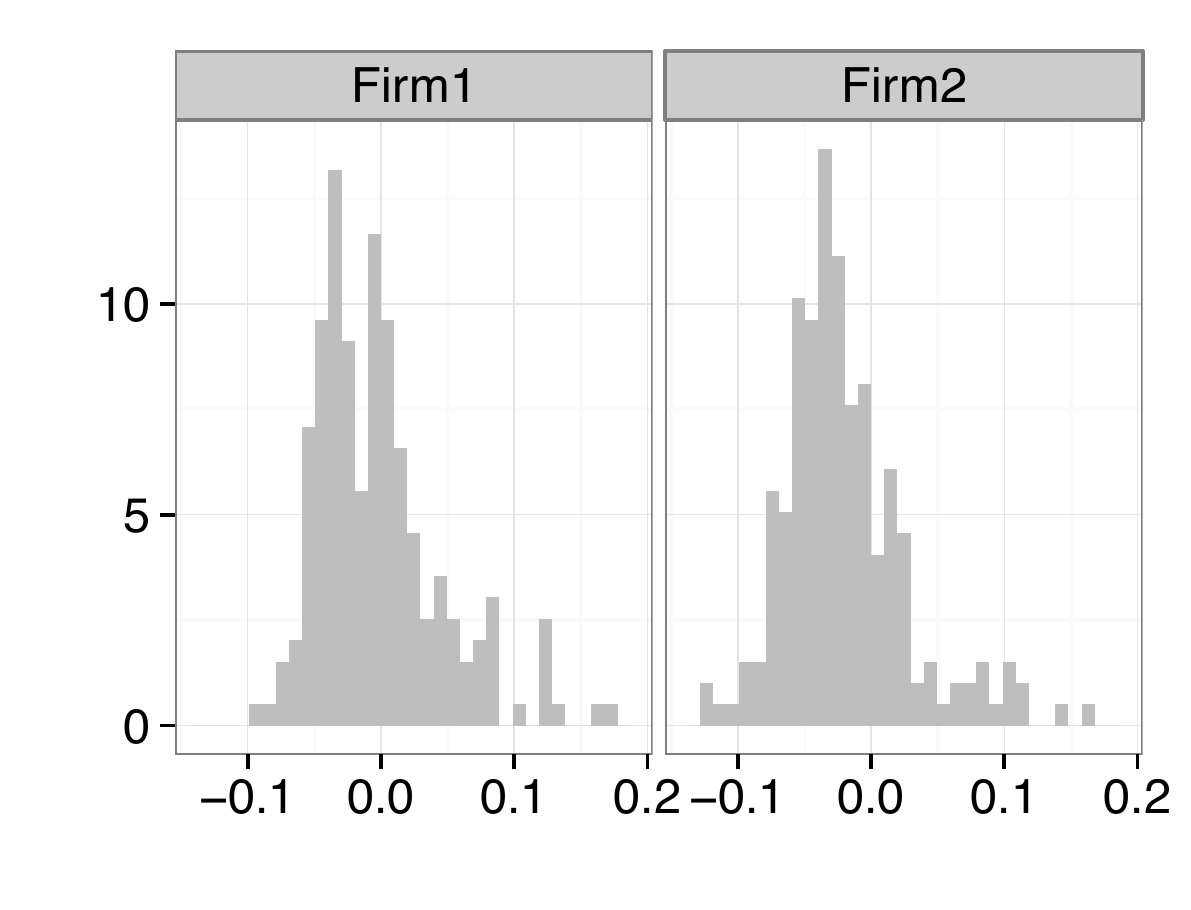}
\end{minipage}
\end{center}
\caption{ \label{fig:DemandHists}
Bertrand-Nash example of Section~\ref{sec:DemandEstimation2}.  The left panel shows the out-of-sample approximation error.  The right panel shows the out-of-sample prediction error.}
\end{figure}   

The maximal residual on the second half of the training set was $z_N \approx 1.06$, indicated by the dotted line in the left panel.  By Theorem~\ref{thm:Campi}, we would expect that with at least 90\% probability with respect to the data sampling, a new point would not be an $1.06$-equilibrium with probability at most $.21$.  Our out-of-sample estimate of this probability is $.025$.  In other words, our estimator has much stronger generalization than predicted by our theorem.  %%VG Add the correct computation here from the Campi result %%
At the same time, our estimator yields reasonably good predictions.  The mean out-of-sample prediction error is $(-.002, 0.02)$ with standard deviation $(.048, .047)$.

Finally, we fit our a nonparametric estimator to this data, using a Gaussian kernel.  We again tune the parameter $c$ and regularization constant $\lambda$ by cross-validation.  The resulting fit is shown in Figure~\ref{fig:FittedDemands} (black line, square markers), along with the corresponding ambiguity set.  We chose the value of $\kappa$ to be twice the standard deviation of the $\ell_1$-norm of the residuals, estimated by cross-validation as discussed in Remark~\ref{rem:Ambiguity} and the end of Sec.~\ref{sec:Generalization}.   The out-of-sample approximation error is similar to the parametric case.  Unfortunately, the fitted function is not monotone, and, consequently, there exist multiple Nash equilibria.  It is thus hard to compare prediction error on the out-of-sample set; which equilibria should we use to predict?  This non-monotonicity is a potential weakness of the nonparametric approach in this example.  

%%VG Add the case where campi can be tight

%%VG Redo the Kernel fit.  From the picture, something is wonky in the normalization.  Also insist on strict dereasingness in constraints...

\subsection{Wardrop Equilibrium}
\label{sec:TrafficEstimation}
Our experiments will use the Sioux Falls network \cite{leblanc1975efficient}, a standard benchmark throughout the transportation literature.  It is modestly sized with 24 nodes and 76 arcs, and all pairs of nodes represent origin-destination pairs.  

We assume that the true function $g( \cdot )$ is given by the U.S. Bureau of Public Roads (BPR) function,
$g(t) = 1 + .15t^4$ which is by far the most commonly used for traffic modeling (\cite{branston1976link}, \cite{bureau1964traffic}).  Baseline demand levels, arc capacities, and free-flow travel times were taken from the repository of traffic problems at \cite{TransportationTestProblems}.  We consider the network structure including arc capacities and free-flow travel times as fixed.  We generate data on this network by first randomly perturbing the demand levels a relative amount drawn uniformly from $[0, 10\%]$.  We then use the BPR function to solve for the equilibrium flows on each arc, $x^*_a$.  Finally, we perturb these true flows by a relative amount, again drawn uniformly from $[0, 10\%]$.  We repeat this process $N=40$ times.  Notice that because both errors are computed as relative perturbations, they both are correlated to the observed values.  We use the perturbed demands and flows as our data set.  

We then fit the function $g$ nonparametrically using \eqref{eq:TrafficQP}, again only using half of the data set.  The use of low order polynomials in traffic modeling is preferred in the literature for a number of computational reasons.  Consequently, we choose $k$ to be a polynomial kernel with degree at most $6$, and tune the choice of $c$ by 5-fold cross-validation, minimizing the approximation error.  The fitted functions for various choices of $d$ are shown in 
 left panel of Figure~\ref{fig:TrafficFit}, alongside the true function.  Notice that the fit is quite stable to choice of class of function, and matches the true function very closely.  In what remains, we focus on our fit of polynomial of degree $3$.  We note, this class does not contain the true BPR function (which has degree 4).  We refit the degree 3 polynomial with the second-half of our training set (not shown).  
%%%%%%%%%%%%%%%%%%%%%%%%%%%%%%%%%%%%%%%%
\begin{figure}
\centering
\begin{minipage}{0.45\textwidth}
	\includegraphics[width = \textwidth]{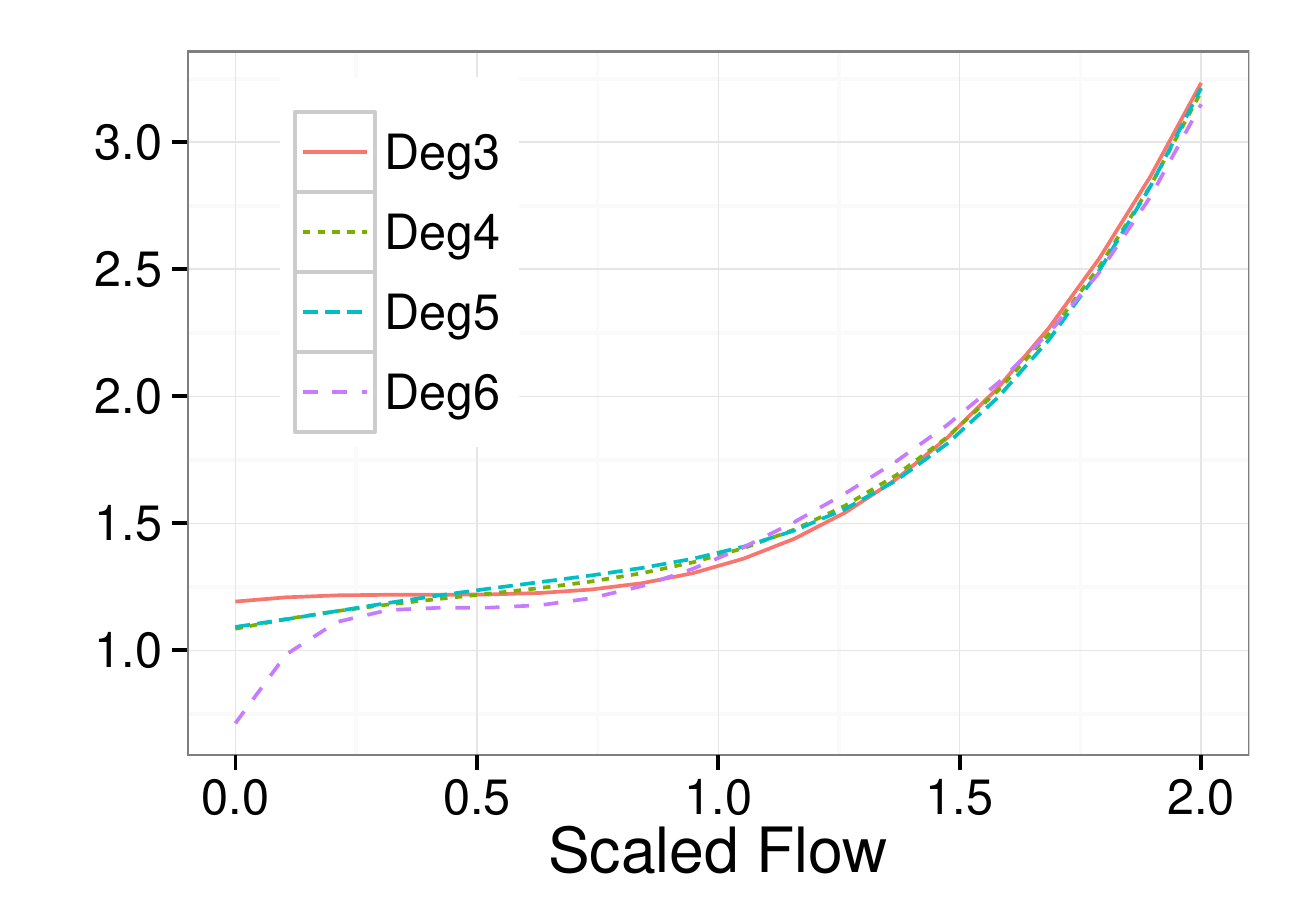}
\end{minipage}
\begin{minipage}{0.45\textwidth}
	\includegraphics[width = \textwidth]{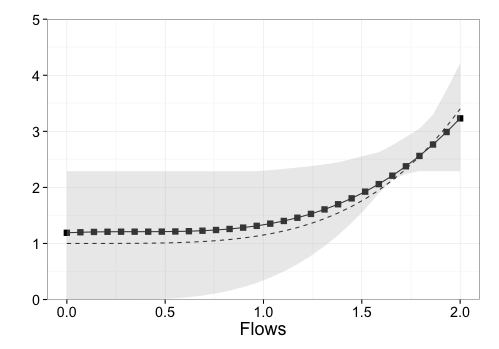}
\end{minipage}
\caption{\label{fig:TrafficFit}  The left panel shows the true BPR function and fits based on polynomials of degrees 3, 4, 5, and 6.  The right panel shows the true BPR function (dashed line), our degree 3 fit (solid black line with markers), and an ambiguity set around this function (grey region).  
}
\end{figure}
%%%%%%%%%%%%%%%%%%%%%%%%%%%%%%%%%%%%%%%%

To assess the quality of our degree $3$ fit, we create a new out-of-sample test-set of size $N_{out}=500$. 
On each sample we compute the approximation error of our fit and the $\ell_2$-norm of the prediction error when predicting new flows by solving the fitted $\VI$. These numbers are large and somewhat hard to interpret.  Consequently we also compute normalized quantities, normalizing the first by the minimal cost of travel on that network with respect to the fitted function and demands, and the second by the $\ell_2$ norm of the observed flows.   Histograms for the normalized quantities are shown in Figure~\ref{fig:TrafficHists}.  The mean (relative) approximation error is 6.5\%, while the mean predictive (relative error) is about 5.5\%.  

The in-sample approximation error on the second-half of the training sample was $z_N \approx 8.14 \times 10^5$.  By Theorem~\ref{thm:Rademacher}, we can compute that with probability at least 90\% with respect to the data sampling, a new data point will be at most a $9.73 \times 10^5$ approximate equilibrium with respect to the fitted function with probability at least 90\%.  A cross-validation estimate of the same quantity is $6.24 \times 10^5$.  Our out of sample estimate of this quantity from the above histogram is $7.78 \times 10^5$.  In other words, the performance of our estimator is again better than predicted by the theorem.  Cross-validation provides a slightly better, albeit biased, bound.     

Finally, as in the previous section, we consider constructing an ambiguity set around the fitted function, selecting $\kappa$ to be two standard deviations as computed by cross-validation.  The resulting envelopes are also shown in the right panel of Figure~\ref{fig:TrafficFit}.  Notice that in contrast to the envelopes of the previous section, they are quite small, meaning we can have relatively high confidence in the shape of the fitted function.  
%%%%%%%%%%%%%%%%%%%%%%%%%%%%%%%%%%%%%%%%
\begin{figure}
\begin{center}
\begin{minipage}{0.4\textwidth}
	\includegraphics[width = \textwidth]{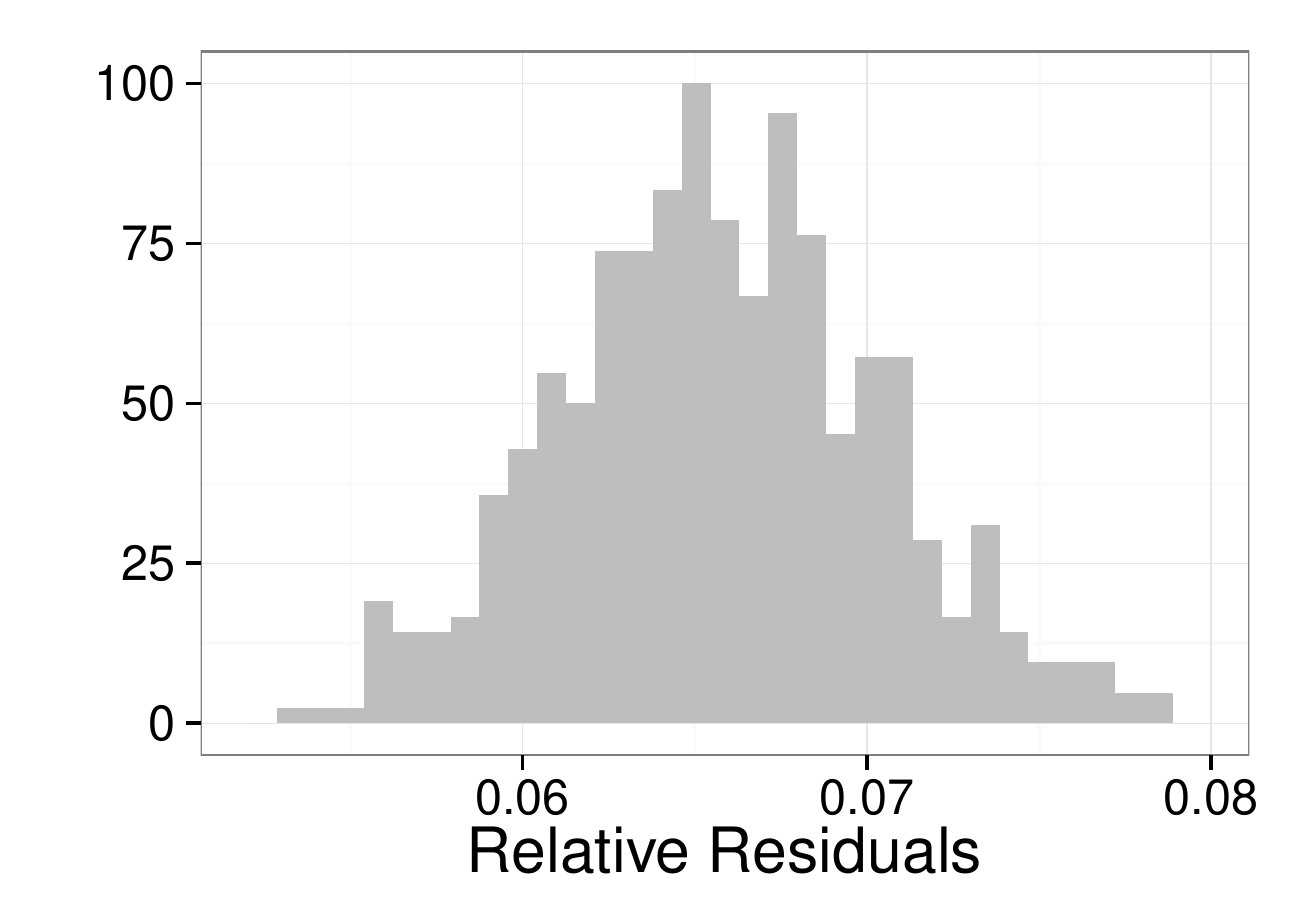}
\end{minipage}
\begin{minipage}{0.4\textwidth}
	\includegraphics[width = \textwidth]{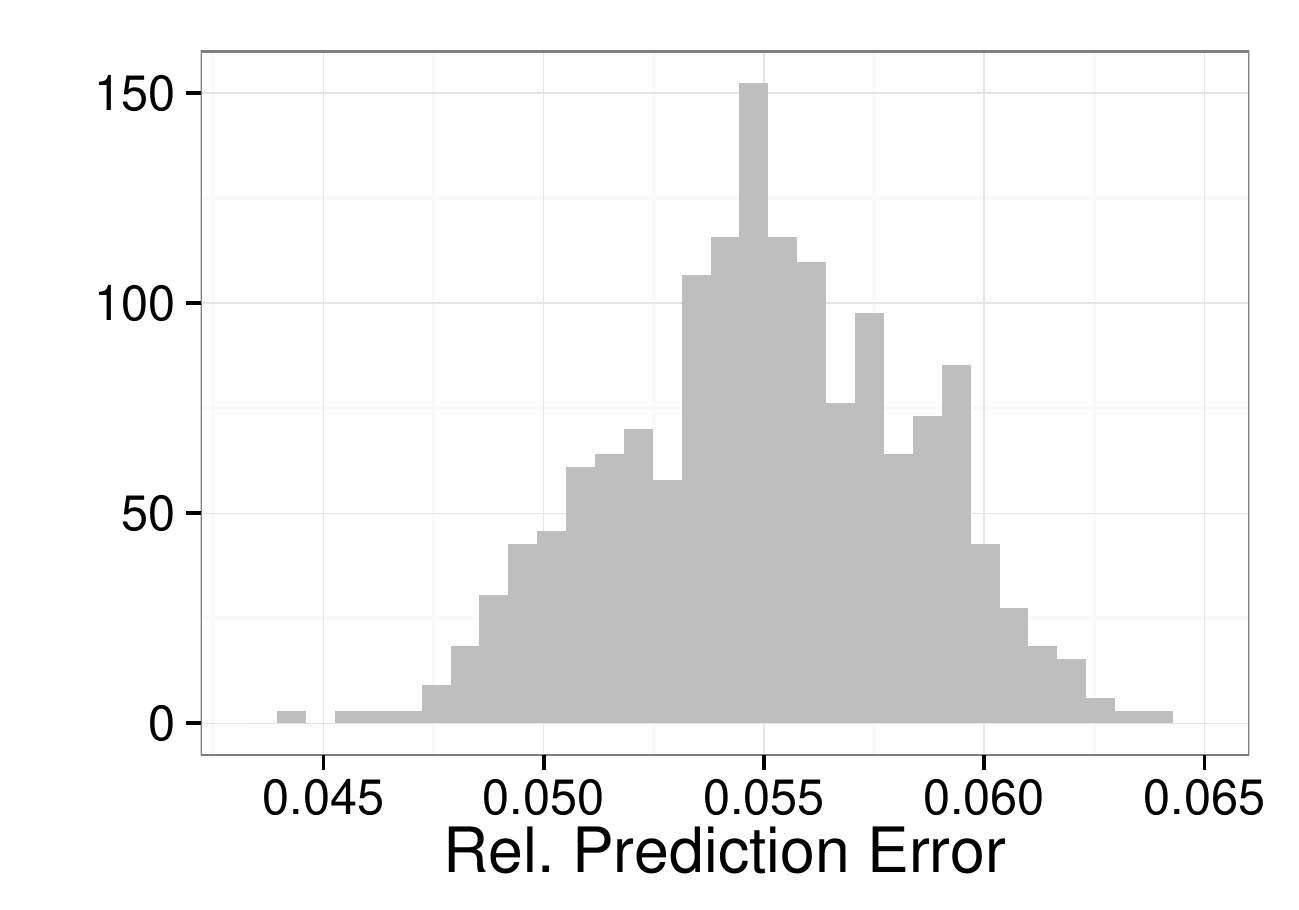}
\end{minipage}
\end{center}
\caption{ \label{fig:TrafficHists}
The left panel shows the histogram of of out-sample approximation errors induced by our nonparametric fit from Section~\ref{sec:TrafficEstimation}.  The right panel shows the norm of the difference of this flow from the observed flow, relative to the norm of the observed flow.}  
\end{figure}   
%%%%%%%%%%%%%%%%%%%%%%%%%%%%%%%%%%%%%%%%

\section{Conclusion}
\label{Sec:Conclusion}

In this paper, we propose a computationally tractable technique for estimation in equilibrium based on an inverse variational inequality formulation.  Our approach is generally applicable and focuses on fitting models with good generalization guarantees and predictive power.  We prove our estimators enjoy both properties and illustrate their usage in two applications -- demand estimation under Nash equilibrium and congestion function estimation under user equilibrium .  Our results suggest this technique can successfully model systems presumed to be in equilibrium and make meaningful predictive claims about them.  

% For one-column wide figures use
%\begin{figure}
% Use the relevant command to insert your figure file.
% For example, with the graphicx package use
%  \includegraphics{example.eps}
% figure caption is below the figure
%\caption{Please write your figure caption here}
%\label{fig:1}       % Give a unique label
%\end{figure}
%
% For two-column wide figures use
%\begin{figure*}
% Use the relevant command to insert your figure file.
% For example, with the graphicx package use
%  \includegraphics[width=0.75\textwidth]{example.eps}
% figure caption is below the figure
%\caption{Please write your figure caption here}
%\label{fig:2}       % Give a unique label
%\end{figure*}
%
% For tables use
%\begin{table}
% table caption is above the table
%\caption{Please write your table caption here}
%\label{tab:1}       % Give a unique label
% For LaTeX tables use
%\begin{tabular}{lll}
%\hline\noalign{\smallskip}
%first & second & third  \\
%\noalign{\smallskip}\hline\noalign{\smallskip}
%number & number & number \\
%number & number & number \\
%\noalign{\smallskip}\hline
%\end{tabular}
%\end{table}
}

\begin{acknowledgements}
%Research partially supported by the NSF under grant EFRI-0735974, by the DOE under grant DE-FG52-06NA27490, by the ARO under grant W911NF-11-1-0227, by the ONR under grant N00014-10-1-0952, and by Citigroup under a grant to the Sloan School of Management.
Research partially supported by the NSF under grants
CNS-1239021 and IIS-1237022, by the ARO under grants
W911NF-11-1-0227 and W911NF-12-1-0390, by the ONR under grant
N00014-10-1-0952, by the NIH/NIGMS under grant GM093147, and by Citigroup under a grant to the Sloan School of Management. 

We would also like to thank the two anonymous reviewers, the Associate editor and the Area editor for their helpful comments on a first draft of this paper.  
\end{acknowledgements}

% BibTeX users please use one of
\bibliographystyle{spmpsci}      % mathematics and physical sciences
\bibliography{InverseVIs_ArXiv2.bib}   % name your BibTeX data base

\appendix

\section{Omitted Proofs} 
\subsection{Proof of Theorem~\ref{Prop:Finite} }
\begin{proof} 
Let $\mathbf{f}^* = (f^*_1, \ldots, f^*_n)$ be any solution.  We will construct a new solution with potentially lower cost with the required representation.  We do this iteratively beginning with $f^*_1$.  

Consider the subspace $\mathcal{T} \subset \H_1$ defined by  $\mathcal{T} = \text{span}(k_{\bx_1}, \ldots, k_{\bx_N})$, and let $\mathcal{T}^\bot$ be its orthogonal complement.
It follows that $f^*_1$ decomposes uniquely into $f_1^*  = f_0 + f_0^\bot$ with $f_0 \in\mathcal{T}$ and $f_0^\bot \in \mathcal{T}^\bot$.  Consequently,
\begin{align*}
f^*_1(\bx_j) &= \langle k_{\bx_j}, f^*_1 \rangle, &\text{(by \eqref{eq:RepProp})}
\\
& =\langle k_{\bx_j} , f_0 \rangle + \langle k_{\bx_j}, f_0^\bot \rangle
\\
& =\langle k_{\bx_j} , f_0 \rangle &\text{(since $f_0^\bot \in \mathcal{T}^\bot$)}
\\
& = f_0(\bx_j) &\text{(by \eqref{eq:RepProp})}.
\end{align*}  
Thus, the solution $\mathbf{f} = (f_0, f^*_2, \ldots, f^*_n)$ is feasible to \eqref{eq:NonParametricConstrained}.  Furthermore, by orthogonality
$
\| f^*_1 \|_{\H_1} = \| f_0 \|_{\H_1} + \| f_0^\bot \|_{\H_1} \geq \| f_0 \|_{\H_1}.
$
Since the objective is non-decreasing in $\| f_1 \|_\H$, $\mathbf{f}$ has an objective value which is no worse than $\mathbf{f}^*$.  We can now proceed iteratively, considering each coordinate in turn.  After at most $n$ steps, we have constructed a solution with the required representation.
\qed
\end{proof}

\subsection{Proof of Theorem~\ref{thm:QP}}
{
\blockblue
\begin{proof}
Suppose Problem~\eqref{eq:QP} is feasible and let $\balpha$ be a feasible solution.  Define $\mathbf{f}$ via eq.~\eqref{eq:KernelExpansion}.  It is straightforward to check that $\mathbf{f}$ is feasible in Problem~\eqref{eq:NonParametricConstrained} with the same objective value.

On the other hand, let $\mathbf{f}$ be some feasible solution to Problem~\eqref{eq:NonParametricConstrained}.  By Theorem~\ref{Prop:Finite}, there exists $\balpha$ such that
$
f_i( \bx_j) = \be_i^T \balpha \bK \be_j,
$ and 
$\| f_i \|^2_\H = \be_i^T\balpha \bK \balpha^T \be_i.
$
It straightforward to check that such $\balpha$ is feasible in Problem~\eqref{eq:QP} and that they yield the same objective value.  Thus, Problem~\eqref{eq:NonParametricConstrained} is feasible if and only if Problem~\eqref{eq:QP} is feasible, and we can construct an optimal solution to Problem~\eqref{eq:NonParametricConstrained}  from an optimal solution to Problem~\eqref{eq:QP} via \eqref{eq:KernelExpansion}.
\qed
\end{proof}
}
{\blockblue
\subsection{Proof of Theorem~\ref{thm:Campi}}
\begin{proof}
As mentioned in the text, the key idea in the proof is to relate \eqref{eq:ParametricInvVI} with a randomized uncertain convex program.  To this end, notice that if $z_N, \btheta_N$ are an optimal solution to \eqref{eq:ParametricInvVI} with the $\ell_\infty$-norm, then
$(z_N, \btheta_N) \in \bigcap_{j=1}^N \mathcal{X}( \bx_j, \bA_j, \bb_j, C_j)$ where
\[
\mathcal{X}( \bx, \bA, \bb, C) = \left\{ z, \btheta \in \Theta : \exists \by\in \R^m \text{ s.t. } \bA^T \by \leq \mathbf{f}(\bx, \btheta), \ \ 
\bx^T\mathbf{f}(\bx, \btheta) - \bb^T\by \leq z \right\}.
\]
The sets $\mathcal{X}( \bx_j, \bA_j, \bb_j, C_j)$ are convex.  Consider then the problem 
\[
\min_{z \geq 0, \btheta} \quad z
\ \
\text{s.t.} \quad (z, \btheta) \in \bigcap_{j=1}^N \mathcal{X}( \bx_j, \bA_j, \bb_j, C_j) .
\]
This is exactly of the form Eq. 2.1 in \cite{campi2008exact}.  Applying Theorem 2.4 of that work shows that with probability $\beta(\alpha)$ with respect to the sampling, the ``violation probability" of the pair $(z_N, \theta_N)$ is a most $\alpha$.  In our context, the probability of violation is exactly the probability that $(\tilde{\bx}, \tilde{\bA}, \tilde{\bb}, \tilde{C})$ is not a $z_N$ approximate equilibria.  This proves the theorem.
%
%In \cite{campi2} it is shown that the there are at most $\text{dim}(\btheta) + 1$ of the $N$ data points $(\bx_j, \bA_j, \bb_j, C_j)$  in \eqref{eq:ParametricInvVI} such that removing that data point would change the solution $\btheta$.  
%In \cite{campi2008exact}, the authors show that the given bound is tight for the class of fully-supported problems, i.e. problems where 
%there are exactly $\text{dim}(\btheta) +1$ such points.  Thus, it suffices to construct such an instance to prove the last portion of the theorem.  To this end suppose that $\tilde{\bA} = I$ almost surely and $\tilde{\bb} \geq \bzero$ almost surely under $\P$, $
%\Theta  = \R^{n}$, and $\mathbf{f}(\bx, \btheta) = \btheta$.  Then, 
%\[
%\mathcal{X}( \bx, \bA, \bb, C) = \left\{ z, \btheta: 
%\ \btheta^T(\bx - \bb) \leq z \right\}.
%\]
%In other words, \eqref{eq:ParametricInvVI} can be rewritten as the following linear optimization problem:
%\begin{align*} \min_{z \geq 0, \btheta}  z
%\text{s.t.} \quad \btheta^T (\bx_j - \bb_j) \leq z \quad j = 1, \ldots, N.
%\end{align*}
%With probability 1, this linear optimization problem is bounded and feasible, and no $\text{dim}(\btheta) + 1$ of the vectors $\bx_j -\bb_j$ are collinear.  Consider the first $\text{dim}(\btheta)$ such vectors.  By inspection, an optimal solution is given by $z=0$, and the unique solution to $\btheta^T (\bx_j - \bb_j) = 0$ for $j=1, \ldots, \text{dim}(\btheta) + 1$.  
%\qed
\end{proof}
Observe that the original proof in \cite{campi2008exact} requires that the solution $\btheta_N$ be unique almost surely.  However, as mentioned on pg. 7 discussion point 5 of that text, it suffices to pick a tie-breaking rule for the $\btheta_N$ in the case of multiple solutions.  The tie-breaking rule discussed in the main text is one possible example.  

\subsection{Proof of Theorem~\ref{thm:Rademacher}} 
We require auxiliary results.  Our treatment closely follows \cite{bartlett2003rademacher}.
Let $\zeta_1, \ldots, \zeta_N$ be i.i.d.  For any class of functions $\mathcal{S}$, define the empirical Rademacher complexity $\mathcal{R}_N(\mathcal{S})$ 
by 
\begin{align*}
\mathcal{R}_N(\mathcal{S}) = 
\E \left[ \sup_{f \in \mathcal{S}} \frac{2}{N}  \left| \left. \sum_{i=1}^N \sigma_i f(\zeta_i) \right|  \right | \zeta_1, \ldots, \zeta_N \right] , 
\end{align*}
where $\sigma_i$ are independent uniform $\{ \pm 1 \}$-valued random variables.
Notice this quantity is random, because it depends on the data $\zeta_1, \ldots, \zeta_N$.  
%Consequently, define the expected Rademacher complexity $\mathcal{R}(\mathcal{S}) \equiv \E[ \mathcal{R}_N(\mathcal{S}) ] $.

Our interest in Rademacher complexity stems from the following lemma.  
\begin{lemma} \label{lemma:fundamental}
Let $\mathcal{S}$ be a class of functions whose range is contained in $[0, M]$.  Then, for any $N$, and any $0 < \beta < 1$, with probability at least $1-\beta$ with respect to $\P$, every $f \in \mathcal{F}$ simultaneously satisfies
\begin{equation}
\E[ f( \zeta ) ] \leq \frac{1}{N} \sum_{i=1}^N f(\zeta_i) + \mathcal{R}_N(\mathcal{S}) + \sqrt{ \frac{8 M \log (2/\beta) }{N}} 
\end{equation}
\end{lemma}
\begin{proof}
The result follows by specializing Theorem 8 of \cite{bartlett2003rademacher}.  Namely, using the notation of that work, let $\phi(y, a) = \mathcal{L}(y, a) = a / M$, $\delta = \beta$ and then apply the theorem.  Multiply the resulting inequality by $M$ and use Theorem 12, part 3 of the same work to conclude that $M \mathcal{R}_N( M^{-1} \mathcal{S} ) = \mathcal{R}_N(\mathcal{S})$ to finish the proof.  
\end{proof}
\begin{remark}
The constants in the above lemma are not tight.  Indeed, modifying the proof of Theorem 8 in \cite{bartlett2003rademacher} to exclude the centering of $\phi$ to $\tilde{\phi}$, one can reduce the constant $8$ in the above bound to $2$.  For simplicity in what follows, we will not be concerned with improvements at constant order.
\end{remark}
\begin{remark}
Lemma~\ref{lemma:fundamental} relates the empirical expectation of a function to its true expectation. If $f \in \mathcal{S}$ were fixed a priori, stronger statements can be proven more simply by invoking the weak law of large numbers. 
The importance of Lemma~\ref{lemma:fundamental} is that it asserts the inequality holds uniformly for all $f \in \mathcal{S}$.  This is important since in what follows, we will be identifying the relevant function $f$ by an optimization, and hence it will not be known to us a priori, but will instead depend on the data.  
\end{remark}

Our goal is to use Lemma~\ref{lemma:fundamental} to bound the 
$\E[\epsilon(\mathbf{f}_N, \tilde{\bx}, \tilde{\bA}, \tilde{\bb}, \tilde{C})]$.  To do so, we must compute an upper-bound on the Rademacher complexity of a suitable class of functions.  As a preliminary step,
\begin{lemma} For any $\mathbf{f}$ which is feasible in \eqref{eq:ParametricInvVI} or \eqref{eq:NonParametricRad}, we have
\begin{equation} \label{eq:UpperBoundEpsilon}
\tilde{\epsilon}(\mathbf{f}, \tilde{\bx}, \tilde{\bA}, \tilde{\bb}, \tilde{C}) \leq \overline{B} \quad \text{ a.s. }
\end{equation}
\end{lemma}
\begin{proof}
Using strong duality as in Theorem~\ref{thm:Aghassi}, 
\begin{equation} \label{eq:DefEpsilon}
\tilde{\epsilon}(\mathbf{f}, \tilde{\bx}, \tilde{\bA}, \tilde{\bb}, \tilde{C}) = 
\max_{\bx \in \tilde{\F}} (\tilde{\bx} - \bx)^T \mathbf{f}( \tilde{\bx} )  
\leq 2R \sup_{\tilde{\bx} \in \tilde{\F}} \| \mathbf{f}(\tilde{\bx}) \|_2,
\end{equation}
by \assref{ass:boundedFeas}.  For Problem~\eqref{eq:ParametricInvVI}, the result follows from the definition of $\overline{B}$.  For Problem~\eqref{eq:NonParametricRad}, observe that
for any $\tilde{\bx} \in \tilde{\F}$, 
\begin{equation} \label{eq:boundonf}
| f_i (\tilde{\bx}) |^2 = \langle f_i, k_{\tilde{\bx}}  \rangle^2 \leq \| f_i \|_\H^2 \sup_{\| \bx \|_2 \leq R } k(\bx, \bx) = \| f_i \|_\H^2 \overline{K}^2
\leq \kappa_i^2 \overline{K}^2,
\end{equation}
where the middle inequality follows from Cauchy-Schwartz.  Plugging this into Eq.~\eqref{eq:DefEpsilon} and using the definition of $\overline{B}$ yields the result.  
\end{proof}
Now consider the class of functions
\[
F = \begin{cases} \Big\{ (\bx, \bA, \bb, C) \mapsto \epsilon(\mathbf{f}, \bx, \bA, \bb, C) : \mathbf{f} = \mathbf{f}(\cdot, \btheta), \btheta \in \Theta \Big\} & \text{ for Problem~\eqref{eq:ParametricInvVI}}
\\
\Big\{ (\bx, \bA, \bb, C) \mapsto \epsilon(\mathbf{f}, \bx, \bA, \bb, C) : f_i \in \H, \  \ \| f_i \|_\H \leq \kappa_i \ \  i=1, \ldots, N \Big\}
& \text{ for Problem~\eqref{eq:NonParametricRad}. }
\end{cases}
\]
%in the case of Problem~\eqref{eq:ParametricInvVI} and 
%\[
%F = \left\{ (\bx, \bA, \bb, C) \mapsto \epsilon(\mathbf{f}, \bx, \bA, \bb, C) : f_i \in \H, \  \ \| f_i \|_\H \leq \kappa_i \ \  i=1, \ldots, N \right\}
%\]
%in the case of Problem~\eqref{eq:NonParametricRad}.  
\begin{lemma}
\[
\mathcal{R}_N(F) \leq \frac{2\overline{B}}{\sqrt{N}}
\]
\end{lemma}
\begin{proof}
We prove the lemma for Problem~\eqref{eq:ParametricInvVI}.  The proof in the other case is identical.  Let 
$
\mathcal{S} =  \{ \mathbf{f}( \cdot, \btheta) : \btheta \in \Theta \}$.  
Then, 
\begin{align*}
\mathcal{R}_N(F) &= \frac{2}{N} \E\left[ \sup_{f \in \mathcal{S}} 
\left.\left| \sum_{j=1}^N \sigma_j \epsilon(\bx_j, \bA_j, \bb_j, C_j) \right| \right| (\bx_j, \bA_j, \bb_j, C_j)_{j=1}^N \right]
\\
&\leq \frac{2 \overline{B}}{N} \E[ (\sum_{j=1}^N \sigma_j^2 )^{\frac{1}{2}}]  &&(\text{using \eqref{eq:UpperBoundEpsilon}})
\\
& \leq \frac{2 \overline{B}}{N} \sqrt{ \E[  \sum_{j=1}^N \sigma_j^2 ] }  &&(\text{Jensen's inequality})
\\
&=\frac{2 \overline{B}}{\sqrt{ N }} && ( \text{$\sigma_j^2 = 1$ a.s.}). 
\end{align*}
\end{proof}

We are now in a position to prove the theorem.  
\begin{proof}[Theorem~\ref{thm:Rademacher} ]
Observe that $z_N  = \frac{1}{N} \sum_{j=1}^N (\epsilon( \mathbf{f}_N, \bx_j, \bA_j, \bb_j, C_j))^p$. 
Next, the function $\phi(z) = z^p$ satisfies $\phi(0) = 0$ and is Lipschitz with constant $L_\phi = p \overline{B}^{p-1}$ on the interval $[0, \overline{B} ]$.  Consequently, 
from Theorem 12 part 4 of \cite{bartlett2003rademacher}, 
\begin{align*}
\mathcal{R}_N( \phi \circ F ) &\leq2 L_\phi  \mathcal{R}_N(F)
\\
&\leq 2p \overline{B}^{p-1} \frac{2\overline{B}}{\sqrt{N}}
\\
&= \frac{4p \overline{B}^p }{\sqrt{N}}.
\end{align*}
Now applying Lemma~\ref{lemma:fundamental} with $\zeta \rightarrow (\tilde{\bx}, \tilde{\bA}, \tilde{\bb}, \tilde{C})$, $f(\cdot) \rightarrow \epsilon( \cdot )^p$, and $M = \overline{B}^p$ yields the first part of the theorem.

For the second part of the theorem, observe that, conditional on the sample, the event $\tilde{\bx}$ is not a $z_N + \alpha$-approximate equilibrium is equivalent to the event that $\epsilon_N > z_N + \alpha$.  Now use Markov's inequality and apply the first part of the theorem.  
\end{proof}

\subsection{Proof of Theorem~\ref{thm:Predictive}}
\begin{proof}
Consider the first part of the theorem.  

By construction, $\hat{\bx}$ solves $\VI(\mathbf{f}(\cdot, \theta_N), \bA_{N+1}, \bb_{N+1}, C_{N+1})$.  The theorem, then, claims that $\bx_{N+1}$ is $\delta^\prime \equiv \sqrt{\frac{z_N}{\gamma}}$ near a solution to this VI.  From Theorem~\ref{thm:PangApprox}, if $\bx_{N+1}$ were not $\delta^\prime$ near a solution, then it must be the case that $\epsilon( \mathbf{f}(\cdot, \theta_N), \bx_{N+1}, \bA_{N+1}, \bb_{N+1}, C_{N+1}) > z_N$.  By Theorem~\ref{thm:Campi}, this happens only with probability $\beta(\alpha)$.  

The second part is similar to the first with Theorem~\ref{thm:Campi} replaced by Theorem~\ref{thm:Rademacher}.  

\end{proof}

{ \blockblue
\section{Casting Structural Estimation as an Inverse Variational Inequality}
\label{sec:Casting}
%A similar observation, i.e. that structural estimation estimators can be seen as the solution to certain types of optimization problems, was made in \cite{dube2012improving} and \cite{su2012constrained}.  By contrast, here we 
In the spirit of structural estimation, assume there exists a \emph{true} $\btheta^* \in \Theta$ that generates solutions $\bx_j^*$ to $\VI(\mathbf{f}(\cdot, \theta^*), \bA^*_j, \bb^*_j, C^*_j)$.  We observe $(\bx_j, \bA_j, \bb_j, C_j)$ which are noisy versions of these true parameters.  We additionally are given a precise mechanism for the noise, e.g., that
\begin{align*}
\bx_j = \bx^*_j + \Delta \bx_j, \quad \bA_j = \bA^*_j + \Delta \bA_j, \quad \bb_j = \bb^*_j + \Delta \bb_j, 
\quad C_j = C_j^*,
\end{align*}
where $(\Delta\bx_j, \Delta \bA_j, \Delta \bb_j)$ are i.i.d. realizations of a random vector $(\tilde{\Delta\bx}, \tilde{\Delta \bA}, \tilde{\Delta \bb})$ and $\tilde{\Delta\bx}, \tilde{\Delta \bA}, \tilde{\Delta \bb}$ are mutually uncorrelated.

We use Theorem~\ref{thm:Aghassi} to estimate $\btheta$ under these assumptions by solving
\begin{align} \nonumber
\min_{\by \geq \bzero, \btheta \in \Theta, \Delta \bx, \Delta \bA, \Delta \bb} \quad & \left\| \begin{pmatrix} \tilde{\Delta\bx_j} \\ \tilde{\Delta \bA_k} \\ \tilde{\Delta \bb_j} \end{pmatrix}_{j=1, \ldots, N} \right\|
\\ \nonumber
\text{s.t.} \quad & 
(\bA_j - \Delta \bA_j)^T \by_j \leq_{C_j} \mathbf{f}(\bx_j -  \Delta \bx_j, \btheta), \ j=1, \ldots, N, 
\\ \label{eq:structest}
&
(\bx_j - \Delta \bx_j)^T \mathbf{f}(\bx_j -  \Delta \bx_j, \btheta) = \bb_j^T \by_j, j =1, \ldots, N,
\end{align} 
where $\| \cdot \|$ refers to some norm.  Notice this formulation also supports the case where potentially some of the components of $\bx$ are unobserved; simply replace them as optimization variables in the above.  
In words, this formulation assumes that the ``de-noised" data constitute a perfect equilibrium with respect to the fitted $\btheta$.

We next claim that if we assume all equilibria occur on the strict interior of the feasible region, Problem~\eqref{eq:structest} is equivalent to a least-squares approximate solution to the equations 
$\mathbf{f}(\bx^*) = \bzero$.  Specifically, when $\bx^*$ occurs on the interior of $\F$, the VI condition Eq.~\eqref{eq:DefVI}  is equivalent to the equations
$
\mathbf{f}(\bx^*) = \bzero.
$
At the same time, by Theorem~\ref{thm:Aghassi}, Eq.~\eqref{eq:DefVI} is equivalent to the system \eqref{eq:DualFeas}, \eqref{eq:StrongDuality} with $\epsilon = 0$ which motivated the constraints in Problem~\eqref{eq:structest}.  Thus, Problem~\eqref{eq:structest} is equivalent to finding a minimal (with respect to the given norm) perturbation which satisfies the structural equations.

We can relate this weighted least-squares problem to some structural estimation techniques.  Indeed, \cite{dube2012improving} and \cite{su2012constrained} observed that many structural estimation techniques can be reinterpreted as a constrained optimization problem which minimizes the size of the perturbation necessary to make the observed data satisfy the structural equations, and, additionally, satisfy constraints motivated by orthogonality conditions and the generalized method of moments (GMM).  In light of our previous comments, if we augment Problem~\eqref{eq:structest} with the same orthogonality constraints, and all equilibria occur on the strict interior of the feasible region, the solutions to this problem will coincide traditional estimators.  

Of course, some structural estimation techniques incorporate even more sophisticated adaptations.  They may also pre-process the data (e.g., 2 stage least squares technique in econometrics) incorporate additional constraints (e.g. orthogonality of instruments approach), or tune the choice of norm in the least-squares computation (two-stage GMM estimation).  These application-specific adaptations improve the statistical properties of the estimator given certain assumptions about the data generating process.  What we would like to stress is that, provided we make the same adaptations to Problem~\eqref{eq:structest} -- i.e., preprocess the data, incorporate orthogonality of instruments, and tune the choice of norm -- and provided that all equilibria occur on the interior, the solution to Problem~\eqref{eq:structest} must coincide exactly with these techniques.  Thus, they necessarily inherit all of the same statistical properties.  

Recasting (at least some) structural estimation techniques in our framework facilitates a number of comparisons to our proposed approach based on Problem~\eqref{eq:ParametricInvVI}.  First, it is clear how our perspective on data alters the formulation.  Problem~\eqref{eq:structest} seeks minimal perturbations so that the observed data are exact equilibria with respect to $\btheta$, while Problem~\eqref{eq:ParametricInvVI} seeks a $\btheta$ that makes the observed data approximate equilibria and minimizes the size of the approximation.  Secondly, the complexity of the proposed optimization problems differs greatly.  The complexity of Problem~\eqref{eq:structest} depends on the dependence of $\mathbf{f}$ on $\bx$ and $\btheta$  (as opposed to just $\btheta$ for \eqref{eq:ParametricInvVI}), and there are unavoidable non-convex, bilinear terms like $\Delta\bA_j^T \by_j$.  These terms are well-known to cause difficulties for numerical solvers.  Thus, we expect that solving this optimization to be significantly more difficult than solving Problem~\eqref{eq:ParametricInvVI}.  Finally, as we will see in the next section, Problem~\eqref{eq:ParametricInvVI} generalizes naturally to a nonparametric setting.  
  }

{\blockblue
\section{Omitted Formulations}
\subsection{Formulation from Section~\ref{sec:DemandEstimation}}
\label{app:DemandForm}

Let $\xi^{med}$ be the median value of $\xi$ over the dataset.  Breaking ties arbitrarily, $\xi^{med}$ occurs for some observation $j = j^{med}$.  Let $p_1^{med}, p_2^{med}, \xi^{med}_1, \xi^{med}_2$ be the corresponding prices and demand shocks at time $j^{med}$.  (Recall that in this section $\xi = \xi_1 = \xi_2$.)  These definitiosn make precise what we mean in the main text by ``fixing other variables to the median observation.}
Denote by $\underline{p}_1, \underline{p}_2$ the minimum prices observed over the data set.  

Our parametric formation in Sec.~\ref{sec:DemandEstimation} is
\begin{subequations} \label{eq:DemandForm} 
\begin{align}  
\min_{\by, \bepsilon, \btheta_1, \btheta_2 } \quad & \| \bepsilon \|_\infty 
\label{eq:2NormObj} 
\\ \notag
\text{s.t.} \quad &  \by^j \geq \bzero, \quad j = 1, \ldots, N,
\\ \nonumber
& y_i^j \geq M_i(p_1^j, p_2^j, \xi^j; \btheta_i), \quad i = 1, 2, \ j = 1, \ldots, N, 
\\ \nonumber
& \sum_{i=1}^2 
\overline{p}^j y_i^j -
(p_i^j) M_i(p_1^j, p_2^j, \xi^j; \btheta_i)  \leq \epsilon_j, \quad j = 1, \ldots, N, 
\\ \label{eq:Decrease1}
&M_1(p_1^j, p_2^{med}, \xi^{med}; \btheta_1) \geq  M_1(p_1^k, p_2^{med}, \xi^{med}; \btheta_1), \quad, \forall 1 \leq j,k \leq N \text{ s.t. } p_1^j \leq p_1^k,
\\ \label{eq:Decrease2}
&M_2(p_1^{med}, p_2^j, \xi^{med}; \btheta_2) \geq  M_2(p_1^{med}, p_2^k, \xi^{med}; \btheta_2), \quad, \forall 1 \leq j,k \leq N \text{ s.t. } p_2^j \leq p_2^k,
\\ \label{eq:Scaling1}
&M_1(\underline{p}_1, p_2^{med}, \xi^{med}; \btheta_1) = M^*_1(\underline{p}_1, p_2^{med}, \xi^{med}_1; \btheta^*_1)
\\ \label{eq:Scaling2}
&M_2(p_1^{med}, \underline{p}_2, \xi^{med}; \btheta_2) = M^*_2(p_1^{med}, \underline{p}_2, \xi^{med}_2 ; \btheta^*_2)
\end{align}
\end{subequations}
Here $M_1$ and $M_2$ are given by Eq.~\eqref{eq:MargRevTrue}.  Notice, for this choice, the optimization is a linear optimization problem.

Eqs.~\eqref{eq:Decrease1} and \eqref{eq:Decrease2} constrain the fitted function to be non-decreasing in the firm's own price.  Eqs.~\eqref{eq:Scaling1} and \eqref{eq:Scaling2} are normalization conditions.  We have chosen to normalize the functions to be equal to the true functions at this one point to make the visual comparisons easier.  In principle, any suitable normalization can be used.  

Our nonparametric formulation is similar to the above, but we replace 
\begin{itemize}
\item The parametric $M_1(\cdot, \btheta_1), M_2(\cdot, \btheta_2)$ with nonparametric $M_1(\cdot), M_2(\cdot) \in \mathcal{H}$
\item The  objective by $\| \bepsilon \|_1 + \lambda ( \| M_1 \|_\H + \| M_2 \|_\H )$.
\end{itemize}
By Theorem~\ref{Prop:Finite} and the discussion in Section~\ref{sec:extensions}, we can rewrite this optimization as a convex quadratic program.  

\subsection{Formulation from Section~\ref{sec:DemandEstimation2}}
\label{app:DemandForm2}

Our parametric formulation is nearly identical to the parametric formulation in Appendix~\ref{app:DemandForm}, with the following changes:
\begin{itemize}
	\item Replace Eq.~\eqref{eq:2NormObj} by $\| \bepsilon \|_\infty + \lambda ( \| \btheta_1 \|_1 + \|\btheta_2\|_1)$
	\item Replace the definition of $M_1, M_2$ by Eq.~\eqref{eq:ParametricGuess}.  
\end{itemize}

Our nonparametric formulation is identical to the nonparametric formulation of the previous section.  
}

\end{document}